\documentclass[onefignum,onetabnum]{siamart220329}
\usepackage{graphics,graphicx,epstopdf,url}
\usepackage{arydshln}
\usepackage{color}
\usepackage{geometry}
\usepackage{algorithmic}
\usepackage{url,cases,supertabular}
\usepackage{amssymb,mathtools}
\usepackage{enumerate,makecell}
\usepackage{amsfonts}
\usepackage{graphicx,epstopdf}
\usepackage{color,verbatim}
\usepackage{mathrsfs,subeqnarray,subfigure}
\usepackage{listings,array}
\usepackage{booktabs}
\usepackage{dashrule}
\usepackage{arydshln}
\usepackage{graphicx}
\usepackage{amsfonts}
\usepackage{mathrsfs}
\usepackage{graphicx}  \usepackage{epstopdf}
\usepackage{subfigure}  \usepackage{lineno}
\usepackage{multirow,bm}
\usepackage{color}
\usepackage{algorithm}
\usepackage{algorithmic}
\usepackage{longtable}
\usepackage{mdwtab}
\usepackage{mdwmath}
\usepackage{bbding}
\usepackage{booktabs}
\usepackage{rotating}
\usepackage{enumitem}
\usepackage{lineno}
\usepackage{url}
\usepackage{pdflscape}
\usepackage{makecell}

\numberwithin{equation}{section}

\newtheorem{assumption}[theorem]{Assumption}

\newcommand{\ba}{\begin{array}}
\newcommand{\ea}{\end{array}}

\newcommand{\bit}{\begin{itemize}}
\newcommand{\eit}{\end{itemize}}
\newcommand{\be}{\begin{equation}}
\newcommand{\ee}{\end{equation}}
\newcommand{\bee}{\begin{equation*}}
\newcommand{\eee}{\end{equation*}}
\newcommand{\bea}{\begin{eqnarray}}
\newcommand{\eea}{\end{eqnarray}}

\newcommand{\st}{\mathrm{s.t.}}

\newcommand{\bs}{\mathrm{\bf s}}
\newcommand{\bx}{\mathbf{x}}
\newcommand{\bp}{\mathbf{p}}
\newcommand{\hbx}{\widehat{\mathbf{x}}}
\newcommand{\hbg}{\widehat{\mathbf{g}}}
\newcommand{\hg}{\widehat{g}}
\newcommand{\bbx}{\overline{\mathbf{x}}}
\newcommand{\tbx}{\widetilde{\mathbf{x}}}
\newcommand{\bd}{\mathbf{d}}
\newcommand{\hbd}{\widehat{\mathbf{d}}}
\newcommand{\tbd}{\widetilde{\mathbf{d}}}
\newcommand{\bq}{\mathbf{q}}
\newcommand{\bu}{\mathbf{u}}

\newcommand{\bW}{\mathbf{W}}
\newcommand{\tbW}{\widetilde{\mathbf{W}}}

\newcommand{\Rmn}[1]{\uppercase\expandafter{\romannumeral#1}}
\renewcommand{\theenumi}{\roman{enumi}}

\newcommand{\Rcal}{\mathcal{R}}

\newcommand{\Pcal}{\mathcal{P}}

\numberwithin{equation}{section}

\newcommand{\Mcal}{\mathcal{M}}
\newcommand{\Ncal}{\mathcal{N}}
\newcommand{\Ccal}{\mathcal{C}}

\newcommand{\grad}{\mathrm{grad}}

\newcommand{\R}{\mathbb{R}}

\numberwithin{theorem}{section}
\newcommand{\iprod}[2]{\left \langle #1, #2 \right \rangle }

\usepackage{lipsum}
\usepackage{clrscode}
\usepackage{appendix}
\usepackage{bm}
\usepackage{booktabs}
\usepackage{url}
\usepackage{multirow}
\usepackage{textcomp}
\usepackage{amsmath}
\usepackage{amsfonts}
\usepackage{amssymb}
\usepackage{mathrsfs}
\usepackage{hyperref}
\usepackage{graphicx,graphics,subfigure}
\usepackage{hyperref,url}
\usepackage{epsf,epstopdf}
\usepackage{algorithm}
\usepackage{algorithmic}
\usepackage{bbm}
\usepackage{dsfont}
\usepackage{indentfirst}
\usepackage{pdfpages}
\usepackage{pdflscape}
\usepackage{longtable}
\ifpdf
  \DeclareGraphicsExtensions{.eps,.pdf,.png,.jpg}
\else
  \DeclareGraphicsExtensions{.eps}
\fi
\allowdisplaybreaks[1]

\headers{Improving the communication in decentralized manifold optimization}{Jiang Hu, Kangkang Deng}

\title{Improving the communication in decentralized manifold optimization through single-step consensus and compression}

\author{
Jiang Hu\thanks{Department of Mathematics, University of California, Berkeley, CA 94720, US
(\email{hujiangopt@gmail.com}).}
\and Kangkang Deng\thanks{Corresponding author. Department of Mathematics,  National University of Defense Technology, Changsha, 410073,
China (\email{freedeng1208@gmail.com}).}
}

\usepackage{amsopn}

\ifpdf
\hypersetup{
  pdftitle={Decentralized projected Riemannian gradient method},
  pdfauthor={Jiang Hu and Kangkang Deng}
}
\fi

\allowdisplaybreaks[2]
\begin{document}

\maketitle

\begin{abstract}
We are concerned with decentralized optimization over a compact submanifold, where the loss functions of local datasets are defined by their respective local datasets. A key challenge in decentralized optimization is mitigating the communication bottleneck, which primarily involves two strategies: achieving consensus and applying communication compression. Existing projection/retraction-type algorithms rely on multi-step consensus to attain both consensus and optimality. Due to the nonconvex nature of the manifold constraint, it remains an open question whether the requirement for multi-step consensus can be reduced to single-step consensus.
We address this question by carefully elaborating on the smoothness structure and the asymptotic 1-Lipschitz continuity associated with the manifold constraint. Furthermore, we integrate these insights with a communication compression strategy to propose a communication-efficient gradient algorithm for decentralized manifold optimization problems, significantly reducing per-iteration communication costs.
Additionally, we establish an iteration complexity of $\mathcal{O}(\epsilon^{-1})$ to find an $\epsilon$-stationary point, which matches the complexity in the Euclidean setting. Numerical experiments demonstrate the efficiency of the proposed method in comparison to state-of-the-art approaches.

\end{abstract}


\section{Introduction}
Decentralized optimization has garnered significant interest due to its wide applications in large-scale distributed systems, such as distributed computing, machine learning, control, and signal processing. These systems often involve data spread across numerous agents or nodes, making centralized optimization approaches impractical due to challenges such as limited storage and computational resources. In this paper, we consider the decentralized smooth optimization over a compact submanifold embedded in the Euclidean space,
\be \label{prob:original}
\begin{aligned}
  \min \quad & \frac{1}{n}\sum_{i=1}^n f_i(x_i), \\
  \st \quad & x_1 = \cdots = x_n, \;\; x_i \in \Mcal, \;\; \forall i=1,2,\ldots, n,
\end{aligned}
\ee
where $n$ represents the total number of agents, $f_i$ is the smooth local objective at the $i$-th agent, and $\Mcal$ is a compact smooth embedded submanifold of $\R^{d\times r}$, e.g., the Stiefel manifold ${\rm St}(d,r):=\{ x \in \R^{d\times r} : x^\top x = I_r \}$. Applications of problem \eqref{prob:original} are ubiquitous in various tasks, e.g., principal component analysis \cite{ye2021deepca}, deep neural networks with batch normalization \cite{cho2017riemannian,hu2022riemannian}, and deep neural networks with orthogonal constraints \cite{arjovsky2016unitary,vorontsov2017orthogonality,huang2018orthogonal,eryilmaz2022understanding}. Decentralized optimization in Euclidean space (i.e., $\mathcal{M} = \mathbb{R}^{d \times r}$) has been extensively studied over the past few decades, see, e.g., \cite{bianchi2012convergence,shi2015extra,xu2015augmented,qu2017harnessing,di2016next,tatarenko2017non, wai2017decentralized,yuan2018exact,hong2017prox,zeng2018nonconvex, scutari2019distributed, sun2020improving}.  For problem \eqref{prob:original} where $\Mcal$ is the Stiefel manifold or other general compact submanifolds, decentralized (stochastic) gradient-type methods have also been recently investigated in \cite{chen2021decentralized,wang2022decentralized,deng2023decentralized,chen2023decentralizediclr}.



In decentralized optimization, all nodes are connected through a network topology in which each node communicates and averages locally with its immediate neighbors. 
When applied to large-scale machine learning applications, the primary bottleneck is communication efficiency, owing to the large number of clients involved in the network and the substantial size of machine learning models. In each iteration, the communication cost comprises two components: the number of node communications and the amount of information that needs to be transmitted during each node communication. The first component is primarily determined by the consensus step. The second component can be reduced by transmitting compressed messages, i.e., communication compression. In the following, we will discuss these two components individually.

The goal of the consensus step is to compute an average across all agents. In the Euclidean setting, consensus can be achieved with a single-step communication in each iteration. A notable difficulty in the manifold case is the direct arithmetic average $\frac{1}{n} \sum_{i=1}^n x_i$ can lie outside of the manifold. This necessitates a more sophisticated design for ensuring consensus. Based on the geodesic distance, Shah \cite{shah2017distributed} proposes a decentralized gradient tracking method, where one needs to perform an asymptotically infinite number of consensus steps in each iteration and each consensus step involves calculating the computationally expensive exponential mapping and parallel transport. In the case where $\Mcal$ is the Stiefel manifold, Wang and Liu \cite{wang2022decentralized} employ an inexact augmented Lagrangian function to dissolve the constraint and construct an equivalent unconstrained form of \eqref{prob:original} under a large enough penalty parameter. They then apply decentralized gradient methods in Euclidean space to solve the resulting problem, allowing for single-step consensus. Instead of using augmented Lagrangian-type methods, a more straightforward approach for manifold optimization is to utilize the retraction/projection operator from Riemannian geometry to ensure constraint consistency. These types of methods have been widely studied in the past decades in the manifold optimization community \cite{absil2009optimization,boumal2023introduction,hu2020brief}. Building on this foundation and leveraging the extrinsic representation of the Stiefel manifold, Chen et al. \cite{chen2021local} define a Euclidean-distance-based consensus problem and demonstrate that the Riemannian gradient iteration with multi-step consensus achieves a locally linear convergence rate under an appropriately chosen step size (less than 1).   By utilizing a powerful tool of proximal smoothness, two works \cite{deng2023decentralized,hu2023achieving} generalize the locally linear convergence results to general compact submanifolds of Euclidean space. In this context, both projected gradient descent and Riemannian gradient descent converge linearly to the consensus under the unit step size.  However, it remains unclear whether single-step consensus is permissible within the framework of projected gradient descent and Riemannian gradient descent-based algorithms.

Another approach to reducing communication costs is communication compression, which involves transmitting compressed messages between clients using compression operators. Two common compression methods are quantization and sparsification. Quantization \cite{alistarh2017qsgd,horvoth2022natural,seide20141}
converts an input from a large, potentially infinite set to a smaller set of discrete values, e.g., 1-bit quantization \cite{seide20141} or natural compression \cite{horvoth2022natural}. Conversely, sparsification \cite{wangni2018gradient,stich2018sparsified} drops a certain number of entries to obtain a sparse version for communication, such as top-$K$ compressors \cite{stich2018sparsified}. Both techniques have proven effective in achieving significant communication savings. Although communication compression has been extensively studied in Euclidean setting, see, e.g., \cite{bianchi2012convergence,shi2015extra,xu2015augmented,qu2017harnessing,di2016next,tatarenko2017non, wai2017decentralized,yuan2018exact,hong2017prox,zeng2018nonconvex, scutari2019distributed, sun2020improving}, it has not been considered in decentralized optimization on manifolds. 


\subsection{Contributions}
The goal of this paper is to address the communication bottleneck. 
Our contributions are summarized as follows: 
\begin{itemize}
    \item \textbf{Proving the applicability of single-step consensus.} Due to the nonconvexity of the manifold, the multi-step consensus is required in projection/retraction based algorithm analysis. Surprisingly, by carefully elaborating the asymptotic 1-Lipschitz continuity of the projection, we show that starting with a refined neighborhood, all the iterates generated by the projected gradient descent or the Riemannian gradient descent with the unit step size stay in the neighborhood even if the single-step consensus is adopted. Subsequently, linear convergence is well-established. This answers the question of the applicability of single-step consensus as posed in \cite{chen2021decentralized,chen2021local,deng2023decentralized,hu2023achieving}. 
    
    \item \textbf{A commnuication-efficient decentralized gradient algorithm.}
    In conjunction with single-step consensus, we investigate communication compression for both the consensus update and gradient tracking to design a communication-efficient algorithm for solving \eqref{prob:original}. This approach can be viewed as a combination of the decentralized projected Riemannian gradient tracking method \cite{deng2023decentralized} and communication compression \cite{tang2018communication}. In addition to its low per-iteration communication cost, our proposed algorithm converges to a $\epsilon$-stationary point with an $\mathcal{O}(\epsilon^{-1})$ iteration complexity, matching the best-known iteration complexity of decentralized Riemannian and Euclidean gradient methods. Consequently, the overall communication cost to achieve the same accuracy as existing algorithms for \eqref{prob:original} is significantly reduced. Numerical experiments demonstrate that the proposed method achieves the same accuracy with only 50\% or less of the communication cost of state-of-the-art methods. This is the first distributed algorithm incorporating communication compression for decentralized manifold optimization.
    
\end{itemize}


\subsection{Notation.}
For the compact submanifold $\Mcal$ of $\mathbb{R}^{d \times r}$, we always take the Euclidean metric $\iprod{\cdot}{\cdot}$ as the Riemannian metric. We use $\|\cdot \|$ to denote the Euclidean norm. We denote the $n$-fold Cartesian product of $\Mcal$ as $\Mcal^n = \Mcal \times \cdots \times \Mcal$.
For any $x\in \Mcal$, the tangent space and normal space of $\Mcal$ at $x$ are denoted by $T_x\Mcal$ and $N_x\Mcal$, respectively. For a differentiable function $h: \R^{d\times r} \rightarrow \R$, we denote its Euclidean gradient by $\nabla h(x)$ and its Riemannian gradient by $\grad h(x)$. For a positive integer $n$, define $[n] =\{1, \dots, n\}$. Let $\mathbf{1}_n\in \mathbb{R}^n$ be a vector where all entries are equal to 1. Define $J := \frac{1}{n}\mathbf{1}_n\mathbf{1}_n^{\top}$.   Unless otherwise explicitly defined, we now provide explanations for all lowercase variables used in this paper. Take $x$ as an example, we denote $x_i$ as a local variable at $i$-th agent; $\hat{x} = \frac{1}{n}\sum_{i=1}^nx_i$ is the Euclidean average. Moreover, we use the following bold notations: $$\bx := [x_1^\top, \dots, x_n^\top]^\top \in \R^{(nd) \times r},~~\hat{\bx} := [\hat{x}^\top, \dots, \hat{x}^\top]^\top \in \R^{(nd) \times r},$$ where $\bx$ denotes the collection of all local variables $x_i$ and $\hat{\bx}$ is $n$ copies of $\hat{x}$. When applied to the iterative process, in $k$-th iteration, we use $x_{i,k}$ to denote a local variable at $i$-th agent and $\hat{x}_k = \frac{1}{n}\sum_{i=1}^nx_{k,i}$. Similarly, we also  denote 
$$\bx_k := [x_{1,k}^\top, \dots, x_{n,k}^\top]^\top \in \R^{(nd) \times r},~~~ \hat{\bx}_k = [\hat{x}_k^\top, \dots, \hat{x}_k^\top]^\top \in \R^{(nd) \times r}.$$ 
Other lowercase variables can also be denoted similarly as $x$.  Define the function $f(\bx): = \frac{1}{n}\sum_{i=1}^n f_i(x_i)$. Let $\bW := W \otimes I_d \in \R^{nd \times nd}$, where $\otimes$ denotes the Kronecker product. 
\section{Preliminary}
In this section, we review some basic concepts of the decentralized optimization on manifolds. 
\subsection{Compact submanifold and smoothness of the projection operator}
Embedded submanifolds of Euclidean space, as described in \cite[Section 3.3]{absil2009optimization}, have a topology that coincides with the subspace topology of Euclidean space. We focus on compact smooth embedded submanifolds of Euclidean space, referring to them as compact smooth submanifolds. Examples of such manifolds include the (generalized) Stiefel manifold, the oblique manifold, and the symplectic manifold.  

To design and analyze the decentralized algorithms for optimization on compact smooth submanifolds, it is observed in \cite{deng2023decentralized} that
the smoothness in certain regimes of the projection mapping of $\Mcal$ plays a crucial role. We begin with the concept of proximal smoothness. The distance and the nearest-point projection of a point $y \in \mathbb{R}^{d\times r}$ onto $\Mcal$ are defined as follows:
 $$
 \text{dist}(x,\Mcal) : = \inf_{y\in \Mcal} \| y - x\|, ~~ \text{and}~~ \Pcal_{\mathcal{M}}(x) : = \arg\min_{y\in \mathcal{M}} \| y - x\|,
 $$
 respectively. For any real number $R >0$, we define the $R$-tube around $\mathcal{M}$ as the set:
 $$
 U_{\mathcal{M}}(R): = \{x:~   \text{dist}(x,\mathcal{M}) < R\}.
 $$
A closed set $\mathcal{M}$ is said to be $R$-proximally smooth if the projection $\Pcal_{\mathcal{M}}(x)$ is a singleton whenever $\text{dist}(x,\mathcal{M}) < R$. Following \cite{clarke1995proximal}, an $R$-proximally smooth set $\mathcal{M}$ satisfies: 
\begin{itemize}
    \item[(i)] For any real $\gamma\in (0,R)$, the estimate holds:
    \be \label{eq:lip-proj}
    \left\| \Pcal_{\mathcal{M}} (x) -\Pcal_{\mathcal{M}} (y)\right\| \leq \frac{R}{R-\gamma}\|x - y\|,~~ \forall x,y \in \bar{U}_{\mathcal{M}}(\gamma),
    \ee
    where $\bar{U}_{\Mcal}(\gamma):=\{x: {\rm dist}(x, \Mcal) \leq \gamma \}$. In particular, $\Pcal_{\Mcal}$ is asymptotic 1-Lipschitz as $\gamma \rightarrow 0$. 
    \item[(ii)] For any point $x\in \mathcal{M}$ and a normal vector $v\in N_x\Mcal$, the following inequality holds for all $y\in\Mcal$:
    \be \label{proximally-0}
    \iprod{v}{y-x} \leq \frac{\|v\|}{2R}\|y-x\|^2.
    \ee
\end{itemize}
It has been shown that any compact $C^2$-submanifold is proximally smooth \cite{clarke1995proximal,balashov2021gradient,davis2020stochastic}. For example, the Stiefel manifold is a $1$-proximally smooth set \cite{balashov2021gradient}. As demonstrated in Section \ref{sec:linear-consensus}, the asymptotic 1-Lipschitz continuity of $\Pcal_{\Mcal}$ is crucial for establishing linear convergence under single-step consensus. {Throughout this paper, we assume that the manifold $\Mcal$ in problem \eqref{prob:original} is $R$-proximally smooth.}

Note that proximal smoothness characterizes only the Lipschitz continuity, rather than the differentiability or higher-order smoothness, of $\Pcal_{\Mcal}$. Due to the smoothness structure of $\Mcal$, it is further shown in \cite[Lemma]{foote1984regularity} that $\Pcal_{\Mcal}$ is smooth within $U_{R}(\Mcal)$. Based on this smoothness, the following lemma on the projection operator holds.
\begin{lemma}{\cite[Lemma 3]{deng2023decentralized}}
\label{lemma-project}
Given an $R$-proximally smooth compact submanifold $\Mcal$, for any $x \in\Mcal, u\in \{u\in \mathbb{R}^{d\times r}:\|u\|\leq \frac{R}{2}\}$, 
there exists a constant $Q > 0$ such that
\be\label{eq:projec-second-order1}
\| \Pcal_{\Mcal}(x + u)  - x - \Pcal_{T_x\Mcal}(u) \| \leq Q\|u\|^2.
\ee
\end{lemma}

\subsection{Stationary point}
Let $x_1,\cdots,x_n\in \Mcal$ represent the local copies of each agent. 
Let $\Pcal_{\Mcal}$ be the orthogonal projection onto $\Mcal$. Note that for $\{x_i\}_{i=1}^n \subset \Mcal$,
\[ {\rm argmin}_{y\in\Mcal}\sum_{i=1}^n \|y - x_i\|^2 = \Pcal_{\Mcal}(\hat{x}). \]
Any element $\bar{x}$ in $\Pcal_{\Mcal}(\hat{x})$ is the induced arithmetic mean of $\{x_i\}_{i=1}^n$ on $\Mcal$ \cite{sarlette2009consensus}.
With a slight abuse of notation, we denote $f(z) : = \frac{1}{n}\sum_{i=1}^n f_i(z)$. The $\epsilon$-stationary point of problem \eqref{prob:original} is defined as follows.
\begin{definition} \label{def:station}
The set of points $\{x_1,x_2,\cdots,x_n\} \subset \Mcal$ is called an $\epsilon$-stationary point of \eqref{prob:original} if there exists a $\bar{x} \in \Pcal_{\Mcal}(\hat{x})$ such that
\[ \frac{1}{n}\sum_{i=1}^n \| x_i - \bar{x}\|^2 \leq \epsilon \quad {\rm and} \quad \|\grad f(\bar{x})\|^2 \leq \epsilon. \]
\end{definition}
In the following development, we always assure that $\hat{x} \in \bar{U}_{\Mcal}(\gamma)$. Consequently, $\Pcal_{\Mcal}(\hat{x})$ is a singleton and we have $\bar{x} = \Pcal_{\Mcal}(\hat{x})$. 

\subsection{Communication compression}
Communication efficiency has been a challenge in fast decentralized optimization. To address this, many studies design convergent communication-compressed algorithms, where the local clients only communicated a small amount of quantized information with their neighbors, see \cite{tang2018communication,tang2019deepsqueeze,koloskova2019decentralized,koloskova2019decentralized2,singh2021squarm,zhao2022beer}. 

Now, we give the definition of commonly used compression operators \cite{tang2018communication}.  
\begin{definition} \label{def:compress}
    We say that a mapping $\mathcal{C}: \mathbb{R}^p \rightarrow \mathbb{R}^p$ is a contractive compression operator if for some constant $0<\alpha \leq 1$ it holds that for any $\bx \in \mathbb{R}^p$
    \be \label{eq:compression} \|\mathcal{C}(\mathbf{x})-\mathbf{x}\|^2 \leq(1-\alpha)\|\mathbf{x}\|^2. \ee
\end{definition}

Note that $\alpha = 1$ corresponds to no compression.  Examples of such compression operators include gsgd\_$b$ \cite{alistarh2017qsgd} and top\_$k$ compression \cite{alistarh2018convergence,stich2018sparsified}. 

\section{Convergence of single-step consensus}  \label{sec:linear-consensus}
To achieve the stationarity as defined in Definition \ref{def:station}, the literature \cite{chen2021decentralized,deng2023decentralized} suggest considering the following consensus problem over $\Mcal$: 
\be \label{prob:consensus}
\min_{\bx} \phi(\bx): = \frac{1}{4} \sum_{i=1}^n \sum_{j=1}^n W_{ij}^t\|x_i - x_j\|^2,~ \text{s.t.}~x_i \in \Mcal, i\in [n].
\ee
 The gradient of $\phi(\bx)$ is $\nabla \phi(\bx): = [\nabla \phi_1(\bx)^\top, \nabla \phi_2(\bx)^\top, \cdots, \nabla \phi_n (\bx)^\top ]^\top = (I_{nd} - \bW) \bx$, where $\nabla \phi_i(\bx): = x_i - \sum_{j=1}^n W_{ij}x_{j}, i \in [n]$. It can be seen that $\nabla \phi(\bx) = 0$ if and only if $\bx = \bW \bx$.

To ensure the equivalence between $\bx = \bW \bx$ and $x_1 = \dots = x_n$, it is necessary to impose assumptions on $W$. 
Denote by the undirected agent network $G:=\{ \mathcal{V}, \mathcal{E}\}$, where $\mathcal{V}=\{1,2,\ldots, n\}$ is the set of all agents and $\mathcal{E}$ is the set of edges. Let $W$ be the mixing matrix associated with $G$. Then $W_{ij} = W_{ji}$ and $W_{ij} > 0$ if an edge $(i,j) \in \mathcal{E}$ and otherwise $W_{ij} = 0$. The following assumption on $W$ is often used in the literature \cite{shi2015extra,zeng2018nonconvex,chen2021decentralized}. 
\begin{assumption} \label{assum-w}
     We assume that $W$ is doubly stochastic, i.e., (i) $W=W^{\top}$; (ii) $W_{i j} \geq 0$ and $1>W_{i i}>0$ for all $i, j$; (iii) Eigenvalues of $W$ lie in $(-1,1]$. The second largest singular value $\sigma_2$ of $W$ lies in $\sigma_2 \in[0,1)$.
\end{assumption}

As presented in \cite{sun2020improving,chen2021decentralized,deng2023decentralized}, the key to designing and analyzing the decentralized algorithms for \eqref{prob:original} is to establish the linear convergence of projected/Riemannian gradient methods for solving \eqref{prob:consensus}. However, since problem \eqref{prob:consensus} is nonconvex, such linear convergence usually relies on the more restrictive condition that $\sigma_2(W) \leq \frac{1}{2\sqrt{n}}$. For a $W$ satisfying Assumption \ref{assum-w}, the authors in \cite{chen2021decentralized,deng2023decentralized} use $W^t$ with $t \geq \log_{\sigma_2} \frac{1}{2\sqrt{n}}$ in the construction of \eqref{prob:consensus} instead of $W$ to meet this requirement. They consider the following consensus problem:
$$
\min_{\bx} \phi^t(\bx): = \frac{1}{4} \sum_{i=1}^n \sum_{j=1}^n W_{ij}^t\|x_i - x_j\|^2,~ \text{s.t.}~x_i \in \Mcal, i\in [n].
$$
In this case, the gradient $\nabla \phi_i^t(\bx) = x_i - \sum_{j=1}^n W^t_{ij}x_j$ needs $t$-step consensus, or simply multi-step consensus. It is unclear whether single-step consensus, i.e., $t=1$, can yield the linear convergence result. We will provide a positive answer to it in the remaining parts of this section.

\subsection{Linear convergence of projected gradient descent}
The projected gradient method with step size $\gamma \in (0, 1]$ for solving \eqref{prob:consensus} is given by
\be \label{eq:consensus-pg-iter}
x_{i,k+1} = \Pcal_{\Mcal}\left( x_{i,k} - \gamma \nabla \phi_i(\bx_k) \right) = \Pcal_{\Mcal}\left((1-\gamma) x_{i,k} + \gamma \sum_{j=1}^n W_{ij} x_{j,k} \right),\quad i\in [n].
\ee
We show that the linear convergence still holds in the single-step consensus setting. The main technique utilized is asymptotic 1-Lipschitz continuity of $\Pcal_{\Mcal}$ of compact submanifold $\Mcal$. Without loss of generality, we assume that $\Mcal$ is $R$-proximally smooth. 

\subsubsection{Staying in neighborhood}
The multi-step consensus in \cite{chen2021decentralized,deng2023decentralized} is used to ensure that the iterations stay in a neighborhood where the restricted secant inequality or the Lipschitz continuity holds. Instead of \eqref{eq:consensus-pg-iter}, we consider the projected gradient descent-based decentralized algorithm for \eqref{prob:original} with update 
\be \label{eq:pgalg} x_{i,k+1} = \Pcal_{\Mcal}\left((1-\gamma) x_{i,k} + \gamma \sum_{j=1}^n W_{ij}x_{j,k} + \eta d_{i,k} \right),  \ee
where $d_{i,k}$ is a descent direction, e.g., $d_{i,k} = -\grad f_i(x_{i,k})$, and $\eta > 0$ is a step size. {When $d_{i,k} = 0$, it reduces to \eqref{eq:consensus-pg-iter}.}

To remove the dependence on multi-step consensus, we need a dedicated neighborhood. Here, we use 
\begin{equation}\label{def:N-neiborhood}
    \mathcal{N}(\delta):=\{ \bx\in \Mcal^n: \|\bx - \bar{\bx} \| \leq \delta \},
\end{equation}
where $\delta > 0$ will be specified later. The following lemma shows that
if $\bx_0 \in \Ncal(\delta)$, all subsequent iterates generated by \eqref{eq:pgalg} satisfy $\bx_{k} \in \Ncal(\delta)$ under certain conditions.
\begin{lemma} \label{lem:stay-neighborhood-pgd}
Let $\{x_{i,k}\}$ be generated by the scheme \eqref{eq:pgalg}. Suppose that Assumption \ref{assum-w} holds,  $\bx_0 \in \Ncal(\delta)$ with $\delta < \min\{R\gamma(1 - \sigma_2), \frac{R}{4} \} $, and $\|\eta \bd_{k}\| \leq \min\{ \frac{R}{4}, \frac{R\gamma(1-\sigma_2) - \delta}{2(R-\delta)}\delta\}$. Then, it holds
\be \label{eq:stay} \bx_k \in \mathcal{N}(\delta), \; \forall k \geq 0. \ee
\end{lemma}
\begin{proof}
    Let us prove it by induction. Assume that $\bx_k \in \Ncal(\delta)$. It follows from the definition of $\| \cdot \|$ that $\max_i \|x_{i,k} - \bar{x}_k\| \leq \|\bx_k - \bar{\bx}_k\| \leq \delta$. Then, by the convexity of $\|\cdot \|$, we have
    \[ \| (1-\gamma)x_{i,k} + \gamma \sum_{j=1}^n W_{ij} x_{j,k}  - \bar{x}_k \| \leq \delta, \; \forall i \in [n], \]
    and $\| \hat{x}_k - \bar{x}_k\| \leq \delta$. Furthermore, for all $i \in [n]$,
    \[ \begin{aligned}
        & {\rm dist}\left((1-\gamma)x_{i,k} + \gamma \sum_{j=1}^n W_{ij} x_{j,k} + \eta d_{i,k}, \Mcal \right) \\
        \leq & {\rm dist}\left((1-\gamma)x_{i,k} + \gamma \sum_{j=1}^n W_{ij} x_{j,k}, \Mcal \right) + \|\eta d_{i,k}\|  \leq  \delta + \frac{R}{4} \leq \frac{R}{2}.
    \end{aligned}
     \]
    Then by the $\frac{R}{R-\delta}$- and $2$-Lipschitz continuity of $\Pcal_{\Mcal}$ over $\bar{U}_{\Mcal}(\delta)$ and $\bar{U}_{\Mcal}(\frac{R}{2})$, it holds 
    \be\label{eq:ineq-stay} \begin{aligned}
        \| \bx_{k+1} - \bar{\bx}_{k+1} \| \leq & \| \bx_{k+1} - \bar{\bx}_k \| = \|\Pcal_{\Mcal^n}( (1-\gamma) \bx_k +  \gamma\bW \bx_k + \eta \bd_k) - \Pcal_{\Mcal}(\hat{\bx}_k) \| \\
        \leq & \|\Pcal_{\Mcal^n}( (1-\gamma) \bx_k +  \gamma\bW \bx_k + \eta \bd_k) - \Pcal_{\Mcal^n}( (1-\gamma) \bx_k +  \gamma\bW \bx_k) \| \\
        & + \|\Pcal_{\Mcal^n}( (1-\gamma) \bx_k +  \gamma\bW \bx_k) - \Pcal_{\Mcal^n}(\hat{\bx}_k) \|  \\
        \leq & 2\eta\|\bd_k\| + \frac{R}{R- \delta} \| ((1-\gamma)I + \gamma \bW) \bx_k - \hat{\bx}_k\| \\
        \leq & 2\eta \|\bd_k\| + \frac{R (1 -\gamma+ \gamma \sigma_2)}{R- \delta} \|\bx_k - \hbx_k\| \\
        \leq & 2\eta \|\bd_k\| + \frac{R (1 -\gamma+ \gamma \sigma_2)}{R- \delta} \|\bx_k - \bar{\bx}_k\|,
    \end{aligned} 
    \ee
    where $\Pcal_{\Mcal^n} = \Pcal_{\Mcal} \times \cdots \times \Pcal_{\Mcal}$, the last inequality follows from the optimality of $\hbx_k$.
    Together with the bound $\delta< R\gamma (1-\sigma_2)$ and $\eta\|\bd_k\| \leq \frac{R\gamma(1-\sigma_2) -\delta}{2(R-\delta)}\delta$, \eqref{eq:ineq-stay} implies $\bx_{k+1} \in \mathcal{N}(\delta)$. We complete the proof. 
\end{proof}

Compared with the analysis of \cite[Lemma 3.2]{deng2023decentralized}, the above lemma allows for the use of the single-step consensus. The key advancements include linking the neighborhood size $\delta$ with the second smallest eigenvalue of the mixing matrix, $\sigma_2$, and demonstrating the asymptotic 1-Lipschitz continuity of the projection $\mathcal{P}_{\mathcal{M}}$.

\subsubsection{Linear convergence}
Having established that all iterates remain within the desired neighborhood, we now proceed to demonstrate the linear convergence of the projected gradient descent as outlined in equation \eqref{eq:consensus-pg-iter}.
\begin{theorem} \label{thm:lin-con-pgd}
Let $\{x_{i,k}\}$ be generated by the scheme \eqref{eq:consensus-pg-iter}. Suppose that Assumption \ref{assum-w} holds and $\bx_0 \in \Ncal(\delta)$ with $\delta < \min\{R\gamma(1 - \sigma_2),\frac{R}{4} \} $. Then, $\{\bx_k\}$ converges to the consensus set linearly with rate $\rho_1 :=\frac{R -R\gamma (1 - \sigma_2)}{R - \delta} < 1$. More specifically, we have for any $k \geq 0$,
\[ \|\bx_{k+1} - \bar{\bx}_{k+1}\| \leq \rho_1 \|\bx_k - \bar{\bx}_k\|.
\]
\end{theorem}
\begin{proof}
    It follows from Lemma \ref{lem:stay-neighborhood-pgd} that $\bx_{k} \in \Ncal(\delta)$ for all $i \in [n]$. 
    Then, using $\frac{R}{R-\delta}$-Lipschitz continuity of $\Pcal_{\Mcal}(\cdot)$ over $\bar{U}_{\Mcal}(\delta)$, we have
    \[ \begin{aligned}
        \| \bx_{k+1} - \bar{\bx}_{k+1} \| & \leq \| \bx_{k+1} - \bar{\bx}_k \| = \|\Pcal_{\Mcal^n}( (1-\gamma) \bx_k +  \gamma\bW \bx_k) - \Pcal_{\Mcal^n}(\hat{\bx}_k) \| \\
        & \leq \frac{R}{R- \delta} \| ((1-\gamma)I + \gamma \bW) \bx_k - \hat{\bx}_k\| \\
        & \leq \frac{R (1 -\gamma+ \gamma \sigma_2)}{R- \delta} \|\bx_k - \hat{\bx}_k\| \\
        & \leq \rho_1 \|\bx_k - \bar{\bx}_k\|. 
    \end{aligned} \]
    This completes the proof. 
\end{proof}

\subsection{Linear convergence of Riemannian gradient descent} 
Riemannian gradient descent is another popular method for achieving consensus, where multi-step consensus is necessary to ensure linear convergence \cite{chen2021local,hu2023achieving}.
Now, we proceed to demonstrate the linear convergence of the single-step consensus-based Riemannian gradient descent. Specifically, with step size $\gamma \in (0,1]$, the iterative scheme is 
\be \label{eq:consensus-it-rgd} x_{i,k+1} = \mathcal{R}_{x_{i,k}}\left(- \gamma \grad \phi_i(\bx_k) \right) =\mathcal{R}_{x_{i,k}}\left(- \gamma \Pcal_{T_{x_{i,k} \Mcal}}\left(x_{i,k} - \sum_{j=1}^n W_{ij}x_{j,k} \right)\right), \ee
where $\mathcal{R}_{x_{i,k}}$ denote a retraction operator at $x_{i,k}\in \Mcal$. Correspondingly, the Riemannian gradient descent-based decentralized algorithm for solving \eqref{prob:original} is 
\be \label{eq:rgalg} x_{i,k+1} = \mathcal{R}_{x_{i,k}}\left(- \gamma \Pcal_{T_{x_{i,k} \Mcal}}\left(x_{i,k} - \sum_{j=1}^n W_{ij}x_{j,k} \right) + \eta d_{i,k}\right), \ee
where $d_{i,k} \in T_{x_{i,k}}\Mcal$.  To analyze the linear convergence of the Riemannian gradient method, the following connection between the retraction and the projection is crucial.
\begin{lemma}{\cite[Lemma 3.1]{hu2023achieving}} \label{lem:lip-proj1}
    Let $\mathcal{R}$ be any retraction on $\Mcal$. For any $x \in \mathcal{M}$ and $u \in \mathbb{R}^{d \times r}$, there exists a positive constant $M_1$ such that
    \be \label{eq:lip-proj1}
    \left\| \mathcal{P}_{\mathcal{M}}(x+u)- \Rcal_x(\mathcal{P}_{T_x \mathcal{M}}(u))\right\| \leq M_1 \|u\|^2. 
    \ee
\end{lemma}

\subsubsection{Staying in neighborhood}
By carefully setting the neighborhood size, we can also show that the iterates generated by \eqref{eq:rgalg} stay in the neighborhood. Our proof is again based on the asymptotical 1-Lipschitz continuity. This requires us to connect the retraction $R$ with the projection $\Pcal_{\Mcal}$. Fortunately, the inequality in Lemma \ref{lem:lip-proj1} is sufficient. Specifically, we have the following lemma on staying in the neighborhood. 
\begin{lemma} \label{lem:rg-stay-neighborhood}
    Let $\{x_{i,k}\}$ be the sequence generated by \eqref{eq:rgalg}.
    Suppose that Assumption \ref{assum-w} holds, $x_0 \in \mathcal{N}(\hat{\delta})$ with 
    $$\hat{\delta} \in \left[0, \min\left\{\frac{8M_1 \gamma^2 R +1 - \sqrt{(8M_1 \gamma^2 R -1)^2 + 32 M_1\gamma^2 R (1 -\gamma+ \gamma \sigma_2)}}{16M_1 \gamma^2}, \frac{R}{4} \right\} \right)$$ with $M_1$ from Lemma \ref{lem:lip-proj1}, and $\|\eta \bd_k\| \leq \min  \left \{ \frac{1}{M_1 \gamma^2}, \frac{\hat{\delta}}{4(1- \rho_2)} \right \}$ with $\rho_2 := \frac{R(1 -\gamma+ \gamma \sigma_2)}{R- \hat{\delta}} + 8M_1 \gamma^2 \hat{\delta} < 1$. Then, it holds that for any $k \geq 0$, 
    \[ \bx_k \in \Ncal(\hat{\delta}). \] 
\end{lemma}

\begin{proof}[Proof of Lemma \ref{lem:rg-stay-neighborhood}]
    It follows from Lemma \ref{lem:lip-proj1} that there exists a constant $M_1 > 0$ such that for any $x\in \Mcal$ and $u \in \R^{d\times r}$
    \[ \|\mathcal{R}_{x}\left(\Pcal_{T_x \Mcal}(u) \right) - \Pcal_{\Mcal}(x + u)\| \leq M_1 \|u\|^2.  \]
    Similar to Lemma \ref{lem:stay-neighborhood-pgd}, we use proof by induction. Assume $\bx_k \in \mathcal{N}(\hat{\delta})$. Then, we have $(1-\gamma)x_{i,k} + \gamma \sum_{j=1}^n W_{ij} x_{j,k}  \in \bar{U}_{\Mcal}( \hat{\delta}), \; i = 1, \ldots, n$, $\hat{x}_k \in  \bar{U}_{\Mcal}( \hat{\delta})$, and $(1 - \gamma)x_{i,k} + \gamma \sum_{j=1}^n W_{ij}x_{i,k} \in \bar{U}_{\Mcal}(\frac{R}{2})$.  
    Using the $\frac{R}{R-\hat{\delta}}$- and 2-Lipschitz continuity of $\Pcal_{\Mcal}(\cdot)$ over $\bar{U}_{\Mcal}(\hat{\delta})$ and $\bar{U}_{\Mcal}(\frac{R}{2})$, we have
    \be \label{eq:rgd-stay} \begin{aligned}
        & \| \bx_{k+1} - \bar{\bx}_{k+1} \|  \leq  \| \bx_{k+1} - \bar{\bx}_{k} \| \\
        \leq & \| \Pcal_{\Mcal}( (1- \gamma)\bx_k + \gamma\bW \bx_k + \eta \bd_k ) - \bar{\bx}_k \| + M_1 \gamma^2 \| \bx_k - \bW \bx_k - \eta \bd_k \|^2 \\
        \leq & \|\Pcal_{\Mcal^n}( (1-\gamma) \bx_k +  \gamma\bW \bx_k + \eta \bd_k) - \Pcal_{\Mcal^n}( (1-\gamma) \bx_k +  \gamma\bW \bx_k) \| \\
        & + \|\Pcal_{\Mcal^n}( (1-\gamma) \bx_k +  \gamma\bW \bx_k) - \Pcal_{\Mcal^n}(\hat{\bx}_k) \| + 2M_1\gamma^2 (\|\bx_k - \bW \bx_k \|^2 + \eta^2 \|\bd_k\|^2 )  \\ 
        \leq & 2 \eta \|\bd_k\| +  \frac{R}{R - \hat{\delta}} \| \bW x_k - \hat{\bx}_k\| + 2M_1 \gamma^2 (4\|\bx_k - \hat{\bx}_k\|^2 + \eta^2 \|\bd_k\|^2) \\
        \leq & \left( \frac{R(1 -\gamma+ \gamma \sigma_2)}{R- \hat{\delta}} + 8M_1 \gamma^2 \hat{\delta} \right) \| \bx_k - \bar{\bx}_k\| + 2(M_1 \gamma^2\eta\|\bd_k\| + 1)\eta \|\bd_k\| \\
        \leq & \rho_2 \delta + 4\eta \|\bd_k\| \leq \delta,
    \end{aligned} \ee
    where we use $\|\bx_k - \bW \bx_k\|^2 = \|\bx_k - \hbx_k + \hbx_k - \bW\bx_k\|^2 \leq 2 \|\bx_k -\hbx_k\|^2 + 2\|\bW \bx_k - \hbx_k\|^2 \leq 4 \|\bx_k - \hbx_k\|^2$ in the third inequality. This completes the proof. 
\end{proof} 

\subsubsection{Linear convergence}
We now show that the single-step consensus-based Riemannian gradient update \eqref{eq:consensus-it-rgd} also converges linearly. 
\begin{theorem} \label{thm:lin-con-rgd}
    Let $\{x_{i,k}\}$ be the sequence generated by \eqref{eq:consensus-it-rgd}. 
    Suppose that Assumption \ref{assum-w} holds and $x_0 \in \mathcal{N}(\hat{\delta})$ with $\hat{\delta}$ defined in Lemma \ref{lem:rg-stay-neighborhood}. Then, $\{\bx_k\}$ converges to the consensus set linearly with rate $\rho_2 < 1$ defined in Lemma \ref{lem:rg-stay-neighborhood}. More specifically, we have for any $k \geq 0$,
    \[ \| \bx_{k+1} - \bar{\bx}_{k+1} \| \leq \rho_2 \|\bx_k - \bar{\bx}_k\|.   \]
\end{theorem}
\begin{proof}
    It follows from Lemma \ref{lem:rg-stay-neighborhood} that $\bx_k \in \Ncal(\hat{\delta})$ for any $k \geq 0$. 
    By setting $\eta =0$ in \eqref{eq:rgd-stay}, we have
    \[ \begin{aligned}
        \| \bx_{k+1} - \bar{\bx}_{k+1} \| & \leq \| \bx_{k+1} - \bar{\bx}_{k} \|  \leq \| \Pcal_{\Mcal^n}( (1- \gamma)\bx_k + \gamma\bW \bx_k ) - \bar{\bx}_k \| + M_1 \gamma^2 \| \bx_k - \bW \bx_k \|^2 \\
        & \leq \frac{R}{R - \hat{\delta}} \| \bW x_k - \hat{\bx}_k\| + 4M_1 \gamma^2 \|\bx_k - \hat{\bx}_k\|^2 
         \leq \rho_2 \| \bx_k - \bar{\bx}_k\|.
    \end{aligned} \]
    This completes the proof. 
\end{proof}


Until now, we have shown that both projection/retraction gradient descent methods have local linear convergence under single-step consensus. In what follows, we focus on the projected gradient descent-based algorithm update \eqref{eq:pgalg}. 


\section{A communication-efficient decentralized gradient method}
In this section, we will use the single-step consensus and communication compression to design a communication-efficient decentralized gradient algorithm for solving \eqref{prob:original}. We first give some notations. Denote $\hat{g}_k = \frac{1}{n}\sum_{i=1}^n \grad f_i(x_{i,k})$ and $\hbg_k = [\hat{g}_k^\top,\dots,\hat{g}_k^\top]^\top$ is $n$ copies of $\hat{g}_k$.

\subsection{Decentralized optimization on manifolds with communication compression}

Motivated by the literature \cite{tang2018communication,tang2019deepsqueeze,koloskova2019decentralized,koloskova2019decentralized2,singh2021squarm,zhao2022beer} for the Euclidean setting, we now present our communication-compression algorithms for decentralized manifold optimization in Algorithm \ref{alg}. The compression is not directly applied to the iterates and gradients. Instead, we introduce the auxiliary variables $\tbx_k :=[\tilde{x}_{1,k}^\top, \dots, \tilde{x}_{n,k}^\top]^\top$ and $\tbd_k := [\tilde{d}_{1,k}^\top, \dots, \tilde{d}_{n,k}^\top]^\top$, with $\tilde{x}_{i,k}$ and $\tilde{d}_{i,k}$ stored in the $i$-th agent.  Here, we compress the differences $\bx_{k+1} - \tbx_k$ and $\bd_{k+1} - \tbd_k$. Then, the local agents communicate these compressed versions $\bq_{k+1}$ and $\bp_{k+1}$ with their neighbors. Therefore, we obtain the local averages $\bs_{k+1}$ and $\bu_{k+1}$ through compressed communications, which are used to update the iterations and tracked gradients. We note that by the initialization strategy $\bs_0 = \bW\bx_0, \bu_0 = \bW \bd_0$, it holds that 
\begin{equation}\label{eq:rela-s-x} \bs_{k+1} = \bW \tbx_{k+1} {\rm \;\; and \;\;} \bu_{k+1} = \bW \tbd_{k+1}. \end{equation}
To tackle the manifold constraint, we use the  Riemannian gradient $\grad f(\bx_k)$ and the projection operator $\Pcal_{\Mcal}$ for the feasibility and efficiency of the algorithm.    

We see that each local agent needs to store six matrices, $x_{i,k}, \hat{x}_{i,k}, s_{i,k}$ and $d_{i,k}, \hat{d}_{i,k}, u_{i,k}$. This is a bit more than the usual uncompressed decentralized manifold optimization algorithms \cite{deng2023decentralized}, where only two matrices $x_{i,k}, d_{i,k}$ are stored. This amount of storage increase is neglectable compared with significantly reduced communications.

\begin{algorithm}[htbp] \label{alg}
\caption{Decentralized projected Riemannian gradient tracking method with communication compression (DPRGC) for solving \eqref{prob:original}} \label{alg:drpgd}
\begin{algorithmic}[1]
\REQUIRE  Initial point $\bx_0\in \mathcal{N}$, $\tbx_0 = \bx_0$, $\bs_0 = \bW \bx_0$, $\bd_0 = \grad f(\bx_0)$, $\tbd_0 = \bd_0$, $\bu_0 = \bW \bd_0$, step sizes $\eta > 0$ and $\gamma \in (0,1]$, set $k = 1$.
\WHILE{the stopping condition is not met}
\STATE $\bx_{k+1} = \Pcal_{\Mcal^n}(\bx_k + \gamma(\bs_k - \tbx_k) - \eta \Pcal_{T_{\bx_k} \Mcal^n}(\bd_k))$.
\STATE $\bq_{k+1} = \Ccal(\bx_{k+1} - \tbx_k)$, $\tbx_{k+1} = \tbx_{k} + \bq_{k+1}$. \hfill \Comment{Compression on the iterates}
\STATE $\bs_{k+1} = \bs_k + \bW \bq_{k+1}$. \hfill \Comment{Communication on the iterates}
\STATE $\bd_{k+1}= \bd_{k} + \gamma(\bu_k - \tbd_k) + \grad f(\bx_{k+1}) - \grad f(\bx_k)$. \hfill \Comment{Gradient tracking}
\STATE $\bp_{k+1} = \Ccal(\bd_{k+1} - \tbd_k)$, $\tbd_{k+1} = \tbd_k + \bp_{k+1}$. \hfill \Comment{Compression on the gradients}
\STATE $\bu_{k+1} = \bu_k + \bW\bp_{k+1}$. \hfill \Comment{Communication on the gradients}
\STATE Set $k=k+1$.
\ENDWHILE
\end{algorithmic}
\end{algorithm}

\subsection{Convergence analysis}
In this subsection, we demonstrate the convergence of Algorithm \ref{alg} for the compression operator in Definition \ref{def:compress}. Let us start with the assumptions on the objective function.
\begin{assumption} \label{ass:lip}
    For any $i = 1,\dots, n$, the function $f_i$ is $L_f$-smooth, i.e., for any $x, y \in \R^{d\times r}$,
    \[ \| \nabla f_i(x) - \nabla f_i(y)\| \leq L_f \|x - y\|.  \]
\end{assumption}

It has been shown in \cite{deng2023decentralized} that there exists a constant $L = \max \left\{ L_f + \frac{1}{R}L_g,  L_f + L_gL_\Pcal \right\}$ such that for any $x$
\be\label{eq:quad}
\begin{aligned}
f_i(y) \leq f_i(x) + \iprod{\grad f_i(x)}{y-x} + \frac{L}{2}\|y-x\|^2, \\
\|\grad f_i(x) - \grad f_i(y)\| \leq L \|x -y\|,
\end{aligned}
\ee
where $L_g:=\max_{x \in \Mcal} \|\nabla f(x)\|$ and $L_{\Pcal}$ is the Lipschitz constant of $\Pcal_{T_x\Mcal}$ over $x \in \Mcal$.

Now, we proceed with the convergence with the following steps: 1. characterize the compression approximation error and consensus error, 2. one-step decrease of the objective function $f$, 3. choose proper step size and establish the convergence rate of $\mathcal{O}(1/K)$. 

We first define the following quantities:
\[ \left\{ \begin{aligned}
{\rm Compression~approximation~error:~} & \Omega_1^k = \|\bx_k - \tbx_k\|^2, \;\; \Omega_2^k = \| \bd_k - \tbd_k \|^2 \\
{\rm Consensus~error:~} & \Omega_3^k = \|\bx_k - \bbx_k \|^2, \;\; \Omega_4^k = \| \bd_k - \hbd_k \|^2 \\
{\rm Gradient~norm:~} & \Omega_5^k = \|\hbd_k\|^2. 
\end{aligned} \right.
\]
 Then, we have the following recursion on these errors.
\begin{lemma} \label{lem:recur}
    Suppose that Assumption \ref{ass:lip} holds. Denote $\rho = 1- \sigma_2$ and $\hat{\rho} = 1- \rho_1$ where $\rho_1$ is defined in Theorem \ref{thm:lin-con-pgd}. For Algorithm \ref{alg}, if $\|\bx_k - \bbx_k\| \leq \delta$ with $\delta$ defined in Theorem \ref{thm:lin-con-pgd} and $2\gamma \|\tbx_k - \bx_k\| + \eta \|\bd_k\| \leq \frac{R}{4}$, it holds that
    \begin{align}
        \Omega_1^{k+1} \leq &\left(1-\frac{\alpha}{2} + \frac{96 \gamma^2}{\alpha} \right) \Omega_1^k
           + \frac{96 \gamma^2}{\alpha} \Omega_3^k 
           + \frac{24 \eta^2}{\alpha} \Omega_4^k + \frac{24 \eta^2}{\alpha} \Omega_5^k, \label{recur-1}\\
        \Omega_2^{k+1} \leq  &\frac{288 L^2\gamma^2}{\alpha} \Omega_1^k + \left(1- \frac{\alpha}{2} + \frac{24 \gamma^2}{\alpha} \right) \Omega_2^k  + \frac{288 L^2 \gamma^2}{\alpha} \Omega_3^k  + \frac{24 \gamma^2 + 72L^2 \eta^2}{\alpha} \Omega_4^k  + \frac{72 L^2 \eta^2}{\alpha} \Omega_5^k, \label{recur-2}  \\
        \Omega_3^{k+1}  \leq  &\frac{64}{\hat{\rho}}\gamma^2 \Omega_1^k + \left( 1- \frac{\hat{\rho}}{2}\right) \Omega_3^k + \frac{16}{\hat{\rho}} \eta^2 \Omega_4^k + \frac{16}{\hat{\rho}} \eta^2 \Omega_5^k, \label{recur-3} \\
        \Omega_4^{k+1} \leq & \frac{288 L^2\gamma}{\rho} \Omega_1^k +   \frac{24\gamma}{\rho} \Omega_2^k 
        +  \frac{288 L^2\gamma}{\rho} \Omega_3^k + 
        \left(1- \frac{\gamma \rho}{2} + \frac{72 L^2\eta^2}{\gamma \rho }\right) \Omega_4^k +
        \frac{72 L^2\eta^2}{\gamma \rho } \Omega_5^k. \label{recur-4}
    \end{align}
\end{lemma}
\begin{proof}
    Let us show the recursions of the errors accordingly.
    
     \textbf{Compression approximation error.} It follows from the definition of $\Ccal$ that
     \be  \label{eq:diff-comp-iter}\begin{aligned}
         \| \tbx_{k+1} - \bx_{k+1}\|^2 & = \| \tbx_k + \Ccal(\bx_{k+1} - \tbx_k) - \bx_{k+1} \|^2 \\
         & \leq (1-\alpha)\| \tbx_k - \bx_{k+1} \|^2 \\
         & \leq (1- \alpha)\left[(1+ \beta) \| \tbx_k - \bx_{k} \|^2 + (1-\alpha)(1+ \frac{1}{\beta}) \|\bx_{k+1} - \bx_k\|^2 \right]\\
         & \leq \left(1-\frac{\alpha}{2}\right)\| \tbx_k - \bx_{k} \|^2 + \frac{2}{\alpha} \|\bx_{k+1} - \bx_k\|^2,
     \end{aligned}
       \ee
       where we use Young's inequality in the second inequality and set $\beta = \frac{\alpha}{2}$ in the last inequality. To bound $\|\bx_{k+1} - \bx_k\|$, we use 
       \be \label{eq:diff-iter} \begin{aligned}
           & \| \bx_{k+1} - \bx_k \|^2 = \|\Pcal_{\Mcal^n}(\bx_k + \gamma (\bs_k - \tbx_k) - \eta \Pcal_{T_{\bx_k} \Mcal^n}(\bd_k)) - \bx_k \|^2 \\
           \leq & 4 \| \bx_k + \gamma (\bW - I) \tbx_k - \eta \Pcal_{T_{\bx_k} \Mcal^n}(\bd_k)) - \bx_k \|^2 \\
           = & 4 \| \gamma (\bW - I)(\tbx_k - \bx_k) + \gamma (\bW - I)(\bx_k - \hbx_k) + \eta \Pcal_{T_{\bx_k} \Mcal^n}(\bd_k))\|^2 \\
           \leq & 12\gamma^2 \| (\bW - I)(\tbx_k - \bx_k)  \|^2 + 12 \gamma^2 \|(\bW - I)(\bx_k - \hbx_k)\| + 12 \eta^2 \|\bd_k\|^2 \\
           \leq & 48 \gamma^2 \|\tbx_k - \bx_k \|^2 + 48 \gamma^2 \| \bx_k - \hbx_k \|^2 + 12 \eta^2 \|\bd_k - \hbd_k\|^2 + 12 \eta^2 \|\hbd_k\|^2 \\
           \leq & 48 \gamma^2 \|\tbx_k - \bx_k \|^2 + 48 \gamma^2 \| \bx_k - \bbx_k \|^2 + 12 \eta^2 \|\bd_k - \hbd_k\|^2 + 12 \eta^2 \|\hbd_k\|^2
           ,
       \end{aligned} \ee
       where the first inequality is due to the fact that $\bs_{k} = \bW \tbx_{k}$ and $\| \Pcal_{\Mcal}(x) - y\| \leq \|\Pcal_{\Mcal}(x) - x\| + \|x -y\| \leq 2\|x-y\|$ for any $x \in \R^{d\times r}, y \in \Mcal$, the second equality use $(\bW - I)\hat{\bx}_k = 0$,  the second inequality is from the Young's inequality $\|x + y\|^2 \leq (1+ \beta) \|x\|^2 + (1+ 1/\beta)\|y\|^2$ for any $\beta > 0$, the second last inequality comes from $\sigma_{\max}(W-I) \leq 2$ and $\|\bd_k\|^2 \leq \|\bd_k - \hbd_k\|^2 + \|\hbd_k\|^2$, and the last line is from definition of $\hbx_k$. Plugging \eqref{eq:diff-iter} into \eqref{eq:diff-comp-iter} yields 
       \be \label{eq:diff-comp-iter-f}
       \begin{aligned}
           \| \tbx_{k+1} - \bx_{k+1}\|^2 \leq  &  \left(1-\frac{\alpha}{2} + \frac{96 \gamma^2}{\alpha} \right)\| \tbx_k - \bx_{k} \|^2 
           + \frac{96 \gamma^2}{\alpha} \| \bx_k - \hbx_k \|^2 \\
           & + \frac{24 \eta^2}{\alpha} \|\bd_k - \hbd_k\|^2 + \frac{24 \eta^2}{\alpha} \|\hbd_k\|^2. 
       \end{aligned}
       \ee

        Let us turn to bound $\|\tbd_{k+1} - \bd_{k+1}\|$. Analogous to \eqref{eq:diff-comp-iter}, it holds
       \be \label{eq:diff-comp-grad}
       \begin{aligned}
           \| \tbd_{k+1} - \bd_{k+1} \|^2 & = \|\tbd_{k} - \Ccal(\bd_{k+1} - \tbd_k) - \bd_{k+1} \| \\
           & \leq (1- \alpha)\|\bd_{k+1} - \tbd_k\|^2 \\
           & \leq (1- \frac{\alpha}{2}) \|\tbd_k - \bd_k\|^2 + \frac{2}{\alpha}\| \bd_{k+1} - \bd_k \|^2.
       \end{aligned}
       \ee
       Similarly, we have from the update formula of $\bd_{k+1}$ that
       \be \label{eq:diff-grad} 
        \begin{aligned}
            & \| \bd_{k+1} - \bd_k \|^2 = \|\bd_k + \gamma (\bu_k - \tbd_k) + \grad f(\bx_{k+1}) - \grad f(\bx_{k}) - \bd_k \|^2 \\
           = & \| \gamma (\bW - I)(\tbd_k - \bd_k) + \gamma (\bW - I)(\bd_k - \hbd_k) + \grad f(\bx_{k+1}) - \grad f(\bx_k) \|^2 \\
           \leq & 3\gamma^2 \| (\bW - I)(\tbd_k - \bd_k)  \|^2 + 3 \gamma^2 \|\bW - I)(\bd_k - \hbd_k)\| + 3 L^2 \|\bx_{k+1} - \bx_k \|^2 \\
           \leq & 12 \gamma^2 \|\tbd_k - \bd_k \|^2 + (12 \gamma^2 + 36L^2 \eta^2) \| \bd_k - \hbd_k \|^2 + 144 L^2\gamma^2 \|\bx_k - \tbx_k\|^2 \\
           & + 144L^2 \gamma^2 \|\bx_k - \hbx_k\|^2  + 36L^2 \eta^2 \|\hbd_k\|^2, 
        \end{aligned}
       \ee
       where the first equality uses $\bu_k = \bW \tbd_k$ in \eqref{eq:rela-s-x} and $(\bW - I)\hbd_k = 0$,  the first inequality is from Young's inequality and the Lipschitz continuity of $\grad f$, and the last inequality comes from \eqref{eq:diff-iter} and $\sigma_{\max}(W- I) \leq 2$. Then, plugging \eqref{eq:diff-grad} into \eqref{eq:diff-comp-grad} gives
       \be \label{eq:diff-com-grad-f}
       \begin{aligned}
           \| \tbd_{k+1} - \bd_{k+1} \|^2 & \leq  \left(1- \frac{\alpha}{2} + \frac{24 \gamma^2}{\alpha} \right) \|\tbd_k - \bd_k\|^2  + \frac{24 \gamma^2 + 72L^2 \eta^2}{\alpha} \| \bd_k - \hbd_k \|^2 \\
           & + \frac{288 L^2\gamma^2}{\alpha} \|\bx_k - \tbx_k\|^2 + \frac{288 L^2 \gamma^2}{\alpha} \|\bx_k - \hbx_k\|^2  + \frac{72 L^2 \eta^2}{\alpha} \|\hbd_k\|^2.   
       \end{aligned}   
       \ee
       
    \textbf{Consensus error.} Due to the existence of manifold constraint, we define the manifold average $\bbx_k$ and the consensus error $\|\bx_k - \bbx_k\|$. Denote $\tbW  = (1-\gamma)I + \gamma \bW$. It follows from $\|\bx_k - \bbx_k\| \leq \delta$ and $2\gamma \|\tbx_k - \bx_k\| + \eta \|\bd_k\| \leq R/2 - \delta$ that $ \| \tbW \bx_k - \bbx_k \| \leq \delta$ and 
    \begin{equation}
    \begin{aligned}
     & \| \tbW \bx_k + \gamma (\bW - I) (\tbx_k - \bx_k) + \eta \Pcal_{T_{\bx_k} \Mcal^n}(\bd_k) - \bbx_k \| \\
      & \leq \| \tbW \bx_k - \bbx_k \| + 2\gamma \| \tbx_k - \bx_k \| + \eta \|\bd_k\|      \leq \delta.
      \end{aligned}
    \end{equation}
    This implies that $\tbW \bx_k$ and $\tbW \bx_k + \gamma (\bW - I) (\tbx_k - \bx_k) + \eta \Pcal_{T_{\bx_k} \Mcal^n}(\bd_k) $ belong to $\bar{U}_{\Mcal^n}(\delta): = \bar{U}_{\Mcal}(\delta)\times \cdots \times \bar{U}_{\Mcal}(\delta)$. 
    By the update formula of Line 2 in Algorithm \ref{alg}, we have
    \be  \label{eq:consensus1}
    \begin{aligned}
        & \|\bx_{k+1} - \bbx_{k+1}\|^2 \leq  \| \bx_{k+1} - \bbx_k \|^2 \\
        = & \| \Pcal_{\Mcal^n}(\bx_k + \gamma (\bW - I) \tbx_k + \eta \Pcal_{T_{\bx_k} \Mcal^n}(\bd_k))) - \bbx_k \|^2  \\
        = & \| \Pcal_{\Mcal^n}( \tbW \bx_k + \gamma (\bW - I) (\tbx_k - \bx_k) + \eta \Pcal_{T_{\bx_k} \Mcal^n}(\bd_k))) - \bbx_k  \|^2 \\
        \leq & (1+ \beta)\| \Pcal_{\Mcal^n}(\tbW \bx_k) - \bbx_k \|^2 \\
        & + (1+ \frac{1}{\beta}) \| \Pcal_{\Mcal^n}(\tbW \bx_k) - \Pcal_{\Mcal^n}(\tbW \bx_k + \gamma (\bW - I) (\tbx_k - \bx_k) + \eta \Pcal_{T_{\bx_k} \Mcal^n}(\bd_k)) \|^2, 
    \end{aligned}    
    \ee
    where the first line is due to the optimality of $\bbx_{k+1}$, the fourth line is from Young's inequality. By setting $\beta = \frac{\hat{\rho}}{2}$ and following Lemma \ref{thm:lin-con-pgd} and 
    the 2-Lipschitz continuity of $\Pcal_{\Mcal^n}$ over $\bar{U}_{\Mcal^n}(\delta)$ and $\delta<\frac{R}{2}$, we have that
    \begin{equation}\label{eq:consensus}
        \begin{aligned}
         & \|\bx_{k+1} - \bbx_{k+1}\|^2 \\
         & \leq 
          (1+ \beta)(1 - \hat{\rho}) \|\bx_k - \bbx_k \|^2 + 4(1+ \frac{1}{\beta})\| \gamma (\bW -I)(\tbx_k - \bx_k) + \eta\Pcal_{T_{\bx_k} \Mcal^n}(\bd_k))\|^2 \\
        \leq & \left( 1- \frac{\hat{\rho}}{2}\right) \|\bx_k - \bbx_k\|^2 + \frac{8}{\hat{\rho}}\left(2 \|\gamma (\bW - I) (\tbx_k - \bx_k)\|^2 + 2\|\eta \bd_k\|^2 \right) \\
        \leq & \left( 1- \frac{\hat{\rho}}{2}\right) \|\bx_k - \bbx_k\|^2 + \frac{64}{\hat{\rho}}\gamma^2 \|\tbx_k - \bx_k\|^2 + \frac{16}{\hat{\rho}} \eta^2 \| \bd_k\|^2 \\
        = & \left( 1- \frac{\hat{\rho}}{2}\right) \|\bx_k - \bbx_k\|^2 + \frac{64}{\hat{\rho}}\gamma^2 \|\tbx_k - \bx_k\|^2 + \frac{16}{\hat{\rho}} \eta^2 \| \hbd_k\|^2 + \frac{16}{\hat{\rho}} \eta^2 \| \bd_k - \hbd_k\|^2,  
        \end{aligned}
    \end{equation}   
    Noting the update formula of $\bd_{k+1}$ of Line 5 in Algorithm \ref{alg}, we have
    \be \label{eq:diff-comp-d} \begin{aligned}
        & \| \bd_{k+1} - \hbd_{k+1} \|^2  \leq \|\bd_{k+1} - \hbd_{k} \|^2 \\
        = & \| \bd_k + \gamma(\bu_k - \tbd_k) - \hbd_k + \grad f(\bx_{k+1}) - \grad f(\bx_k) \|^2 \\
        = & \| \tbW \bd_k - \hbd_k + \gamma (\bW - I)(\tbd_k - \bd_k) + \grad f(\bx_{k+1}) - \grad f(\bx_k)   \|^2 \\
        \leq & (1+ \beta) \|\tbW \bd_k - \hbd_k \|^2 + \left( 1 + \frac{1}{\beta}\right)\left( 2\gamma^2 \|(\bW-I)(\tbd_k - \bd_k)\|^2 + 2 L^2\|\bx_{k+1} - \bx_k\|^2  \right) \\
        \leq & \left(1+\beta\right) (1-\gamma \rho) \| \bd_k - \hbd_k \|^2  + \left(1 + \frac{1}{\beta}\right) \left( 8\gamma^2 \|(\tbd_k - \bd_k)\|^2 + 2 L\|\bx_{k+1} - \bx_k\|^2  \right) \\
        \leq & \left(1- \frac{\gamma \rho}{2}\right) \| \bd_k - \hbd_k \|^2  +   \frac{24\gamma}{\rho} \|\tbd_k - \bd_k\|^2 + \frac{6L^2}{\gamma \rho}\|\bx_{k+1} - \bx_k\|^2,
    \end{aligned}
     \ee
     where the first inequality is from the optimality of $\hbd_{k+1}$, the third line is due to the fact $\bu_k = \bW \hbd_k$, the fourth line comes from the Young's inequality, we use $\sigma_{\max}(\tbW - I) \leq (1-\gamma \rho)$ and $ \sigma_{\max}(W - I) \leq 2$ in the fifth line,  we use $\beta = \gamma \rho /2$ and $\gamma \rho < 1$ in the sixth line. Plugging \eqref{eq:diff-iter} into \eqref{eq:diff-comp-d} leads to 
     \be \label{eq:diff-comp-d-f} 
     \begin{aligned}
     \| \bd_{k+1} - \hbd_{k+1} \|^2 \leq & \left(1- \frac{\gamma \rho}{2} + \frac{72 L^2\eta^2}{\gamma \rho }\right) \| \bd_k - \hbd_k \|^2  +   \frac{24\gamma}{\rho} \|\tbd_k - \bd_k\|^2 \\
      & + \frac{288 L^2\gamma}{\rho} \|\tbx_k - \bx_k \|^2 + \frac{288 L^2\gamma}{\rho} \| \bx_k - \bbx_k \|^2 + 
    \frac{72 L^2\eta^2}{\gamma \rho } \|\hbd_k\|^2,
    \end{aligned}
    \ee   
\end{proof}

If the step sizes $\eta$ and $\gamma$ are chosen small enough, the assumptions used Lemma \ref{lem:recur}, $\|\bx_k - \bbx_k\| \leq \delta$ and $2\gamma \|\tbx_k - \bx_k\| + \eta \|\bd_k\| \leq R/4$, will always hold. These two facts are necessary to use the Lipschitz continuity of $\Pcal_{\Mcal}$ in \eqref{eq:consensus}.
\begin{lemma} \label{lem:bound}
    Let $\gamma = C_{\gamma}\alpha$ and $\eta = C_{\eta} \gamma \rho /L$ with some positive constants $C_{\gamma}, C_{\eta}$. For an arbitrary small $\delta > 0$, there exist small enough 
    but fixed $C_{\gamma}$ and $C_{\eta}$, such that for all $k$,
    \[ \| \bx_k - \bbx_k \|^2 \leq \delta^2, \;\; \| \tbx_k - \bx_k \|^2 \leq 2 C_{\gamma} \delta^2, \;\; \|\tbd_k -\bd_k \|^2 \leq \frac{nL_g^2}{\gamma^2},\;\; \|\bd_k - \hbd_k\|^2 \leq \frac{192 n L_g^2}{\gamma^2 \rho^2}.  \]
\end{lemma}
\begin{proof}
    We use proof by induction. For $k=0$, it follows the initialization, we have
    $ \|\bx_0 - \bbx_0 \| = 0, \;\; \| \tbx_0 - \bx_0\| = 0, \;\; \|\bd_0 \| = \|\grad f(\bx_0)\| \leq \sqrt{n} L_g.$
    Assume that 
    \[ \| \bx_k - \bbx_k \|^2 \leq \delta^2, \;\; \| \tbx_k - \bx_k \|^2 \leq 2C_{\gamma} \delta^2, \;\;\|\tbd_k -\bd_k \|\leq \frac{nL_g^2}{\gamma^2},\;\; \|\bd_k - \hbd_k\| \leq \frac{192 n L_g^2}{\gamma^2 \rho^2}. \]
    Following  \eqref{eq:diff-comp-d}, we have
    \be \label{eq:consensus-d}
    \begin{aligned}
        \|\bd_k - \hbd_k \|^2 & = \| \tbW \bd_k - \hbd_k + \gamma (\bW - I)(\tbd_k - \bd_k) + \grad f(\bx_{k+1}) - \grad f(\bx_k)   \|^2 \\
        & \leq \left(1- \frac{\gamma \rho}{2}\right) \| \bd_k - \hbd_k \|^2  +   \frac{24\gamma}{\rho} \|\tbd_k - \bd_k\|^2 + \frac{6}{\gamma \rho}\|\grad f(\bx_{k+1}) - \grad f(\bx_k)\|^2 \\
        & \leq \left(1- \frac{\gamma \rho}{2}\right) \| \bd_k - \hbd_k \|^2  +   \frac{24\gamma}{\rho} \|\tbd_k - \bd_k\|^2 + \frac{24n L_g^2}{\gamma\rho}.
    \end{aligned}
    \ee
    Define ${\bf \Omega}^{k+1} := [\Omega_1^k, \Omega_2^k, \Omega_3^k, \Omega_4^k]^\top$. 
    Combining \eqref{eq:consensus-d} with \eqref{recur-1}, \eqref{recur-2}, \eqref{recur-3}, and \eqref{recur-4} gives
    \be \label{recur-matrix} {\bf \Omega}^{k+1} \leq {\bf A} {\bf \Omega}^k + {\bf a}, \ee
    where
    \[
{\bf A} = \begin{bmatrix}
1 - \frac{\alpha}{2} + \frac{96 \gamma^2}{\alpha} & 0 & \frac{96 \gamma^2}{\alpha} & \frac{24 \eta^2}{\alpha} \\
\frac{288 L^2 \gamma^2}{\alpha} & 1 - \frac{\alpha}{2} + \frac{24 \gamma^2}{\alpha} & \frac{288 L^2 \gamma^2}{\alpha} & \frac{24 \gamma^2 + 72L^2 \eta^2}{\alpha} \\
\frac{64 \gamma^2}{\hat{\rho}} & 0 & 1 - \frac{\hat{\rho}}{2} & \frac{16 \eta^2}{\hat{\rho}} \\
0 & \frac{24 \gamma}{\rho} & 0& 1 - \frac{\gamma \rho}{2}
\end{bmatrix}
\]
and ${\bf a} = [\frac{24 \eta^2}{\alpha} \Omega_5^k, \frac{72 L^2 \eta^2}{\alpha} \Omega_5^k, \frac{16 \eta^2}{\hat{\rho}} \Omega_5^k, \frac{24 nL_g^2}{\gamma \rho} ]^\top$. For sufficiently small $C_{\gamma}$ and $C_{\eta}$ and noting $\hat{\rho} = \mathcal{O}(\gamma)$ when $\gamma \rightarrow 0$, it holds 
\[ 1- \frac{\alpha}{2} + \frac{96 \gamma^2}{\alpha} \leq 1 - \frac{\alpha}{4}, \; \frac{288 L^2 \gamma^2}{\alpha} \leq \frac{\alpha  C_\gamma}{4},\; \frac{72L^2\eta^2}{\alpha} \leq \frac{\alpha \eta}{4}, \; \frac{72L\eta^2}{\gamma \rho} \leq \frac{\gamma \rho C_\eta}{4}. \]
It follows from the update rule of $\bd_k$ and \eqref{eq:rela-s-x} that  
\begin{equation}
\begin{aligned}
    & \hat{d}_{k+1}  = \hat{d}_k +  \gamma \frac{1}{n} \sum_{i=1}^n ( u_{i,k} - \tilde{d}_{i,k} )+ \hat{g}_{k+1} - \hat{g}_k  \\
    & = \hat{d}_k +  \gamma \frac{1}{n} \sum_{i=1}^n ( \sum_{j=1}^n W_{ij} \tilde{d}_{j,k} - \tilde{d}_{i,k} )+ \hat{g}_{k+1} - \hat{g}_k  = \hat{d}_k + \hat{g}_{k+1} - \hat{g}_k
    \end{aligned}
\end{equation}
Since the initial strategy $d_{i,0} = \grad f_i(x_{i,0})$, we have that $\hat{d}_k = \hat{g}_k =  \frac{1}{n}\sum_{i=1}^n \grad f_i(x_{i,k})$.  
This implies $\|\Omega_5^k\| \leq n L_g^2$ for any $k$. Then, the inequality \eqref{recur-matrix} gives
\[ \begin{aligned}
    \| \Omega_1^{k+1} \| & \leq \left(1- \frac{\alpha}{4} \right) 2C_\gamma \delta^2 + \frac{\alpha C_{\gamma}}{4} \delta^2 + C_{\eta}\eta(\frac{192 n L_g^2}{\gamma^2 \rho^2}+ n L_g^2) \\
    \|\Omega_2^{k+1}\| & \leq \left(1 -\frac{\alpha}{4}\right) \frac{nL_g^2}{\gamma^2} + \frac{\alpha C_{\gamma}}{4}(2C_\gamma \delta^2 + \delta^2) + \frac{\alpha C_{\gamma}}{2} \frac{192 n L_g^2}{\gamma^2 \rho^2} + \frac{\alpha \eta}{4} n L_g^2, \\
    \|\Omega_3^{k+1}\| & \leq (1 - \frac{\hat{\rho}}{2})\delta^2 + \frac{64\gamma^2 }{\hat{\rho}}2C_\gamma \delta^2 +  \frac{C_{\eta} \hat{\rho}}{2} (\frac{192 n L_g^2}{\gamma^2 \rho^2} + nL_g^2) \\
    \|\Omega_4^{k+1}\| & \leq \left(1- \frac{\gamma \rho}{4} \right) \frac{192 n L_g^2}{\gamma^2 \rho^2} + \frac{24 C_\gamma}{\rho} \frac{nL_g^2}{\gamma^2} + \frac{24n L_g^2}{\gamma \rho}. 
\end{aligned} \]
It is easy to verify that for  small enough $C_\gamma$ and $C_{\eta}$, 
\[ \begin{aligned}
    & \| \bx_{k+1} - \bbx_{k+1} \|^2 \leq \delta^2, \;\; \| \tbx_{k+1} - \bx_{k+1} \|^2 \leq 2 C_{\gamma} \delta^2, \\
    & \|\tbd_{k+1} -\bd_{k+1} \|^2 \leq \frac{nL_g^2}{\gamma^2},\;\; \|\bd_{k+1} - \hbd_{k+1}\|^2 \leq \frac{192 n L_g^2}{\gamma^2 \rho^2}.
\end{aligned}
 \]
\end{proof}

The following inequality on the distance between the Euclidean mean and manifold mean holds. 
\begin{lemma}{\cite[Lemma 4]{deng2023decentralized}} \label{lemma:quadratic}
   For any $\bx\in \Mcal^n$ satisfying $\|x_i - \bar{x}\| \leq \frac{R}{2}$, $i\in [n]$,  we have 
   \be\label{eq:distance-an-rm}
   \|\bar{x} - \hat{x} \| \leq M_2 \frac{\|\bx - \bar{\bx}\|^2}{n},
   \ee
   where $M_2 =\max_{x\in \bar{U}_\Mcal(\frac{R}{2})} \|D^2 \Pcal_{\Mcal}(x) \|_{\rm op}$.
\end{lemma}

Based on the Lipschitz smoothness assumption \ref{ass:lip}, we have the following one-step decrease on $f$. 

\begin{lemma} \label{lem:descent}
    Suppose that Assumptions \ref{assum-w} and \ref{ass:lip} hold. Let $\gamma = C_{\gamma} \alpha$ and $\eta = C_{\eta}\gamma \rho/L$.
    For sufficiently small $C_{\eta}$ and $C_{\gamma}$, we have
    \be \label{eq:descent} 
    \begin{aligned}
       f(\bbx_{k+1}) \leq & f(\bbx_k) - (\eta - (2 Q +  25)L \eta^2)  \|\hg_k\|^2 + \mathcal{D}_1 \frac{1}{n} \| \bx_k - \bbx_k\|^2 \\
       &+ \mathcal{D}_2 \frac{1}{n} \|\bx_{k+1} - \bbx_{k+1}\|^2 + \mathcal{D}_3  \frac{\gamma^2}{n}\|\bx_k - \tbx_k\|^2 + \mathcal{D}_4 \frac{\eta^2}{n}\|\bd_k - \hbd_k\|^2, 
    \end{aligned}
    \ee
    where $\mathcal{D}_1 = 16 QL\gamma^2 + 8 \gamma^2 L + L + \frac{c_0 M_2^2 L}{ \gamma^4} + 192L\gamma^2, \mathcal{D}_2 = \frac{c_0 M_2^2 L}{ \gamma^4}, \mathcal{D}_3 = 16 QL +   200L$, and $\mathcal{D}_4 = 2QL +  25 L $. 
\end{lemma}
\begin{proof}
    It follows from \eqref{eq:quad} that
    \be \label{eq:dec-lip} 
    \begin{aligned}
      &  f(\bbx_{k+1}) \leq  f(\bbx_k) + \iprod{\grad f(\bbx_{k})}{\bbx_{k+1} - \bbx_k} + \frac{L}{2}\|\bbx_{k+1} - \bbx_k\|^2 \\
        = &   f(\bbx_k) + \iprod{\hbg_k}{\bbx_{k+1} - \bbx_k} + \iprod{\grad f(\bbx_{k}) - \hbg_k}{\bbx_{k+1} - \bbx_k} + \frac{L}{2}\|\bbx_{k+1} - \bbx_k\|^2 \\
        \leq & f(\bbx_k) + \iprod{\hbg_k}{\bbx_{k+1} - \bbx_k} + \frac{1}{L}\| \hbg_k - \grad f(\bbx_k)\|^2 + \frac{3L}{4}\|\bbx_{k+1} - \bbx_k \|^2 \\
        \leq & f(\bbx_k) + \iprod{\hbg_k}{\hbx_{k+1} - \hbx_k} + \iprod{\hbg_k}{\bbx_{k+1}- \hbx_{k+1} + \hbx_{k} - \bbx_k}+ L \| \bx_k - \bbx_k\|^2  + \frac{3L}{4}\|\bbx_{k+1} - \bbx_k \|^2 \\
        \leq & f(\bbx_k) + \underbrace{\iprod{\hbg_k}{\hbx_{k+1} - \hbx_k}}_{b_1} + \frac{\eta^2 L}{2}\| \hbg_{k}\|^2 + \underbrace{\frac{L}{\eta^2}(\|\bbx_{k+1} - \hbx_{k+1}\|^2 + \|\hbx_k - \bbx_k\|^2)}_{b_2} \\
        & + L \|\bx_k - \bbx_k \|^2 + \underbrace{\frac{3L}{4}\|\bbx_{k+1} - \bbx_k\|^2}_{b_3},
    \end{aligned}
    \ee
    where the second inequality is from Young's inequality, the third inequality is due to the fact that 
    $$
    \begin{aligned}
    & \|\hbg_k - \grad f(\bbx_k)\|^2 \leq n \| \hat{g}_k - \grad f(\bar{x}_k) \|^2 \\
    & \leq \frac{n}{n} \sum_{i=1}^n \|\grad f_i(x_k) -  \grad f_i(\bar{x}_k) \|^2 \leq L^2 \|\bx_k - \bar{\bx}_k \|^2,
    \end{aligned}
    $$ and the fourth inequality is because of Young's inequality. 

    We turn to bound $b_1, b_2$ and $b_3$, respectively. Denote $v_{i,k} = \Pcal_{T_{x_{i,k}} \Mcal }(d_{i,k})$. 
    \be \label{eq:b1}
    \begin{aligned}
        b_1 = & \iprod{\hg_k}{\frac{1}{n}\sum_{i=1}^n (x_{i,k+1} - x_{i,k} - \gamma (s_{i,k} - \widetilde{x}_{i,k}) + \eta v_{i,k}} - \iprod{\hg_k}{ \eta \frac{1}{n} \sum_{i=1}^n v_{i,k}} \\
        = & \frac{1}{n}\iprod{\hg_k}{\sum_{i=1}^n (x_{i,k+1} - x_{i,k} - \gamma (s_{i,k} - \widetilde{x}_{i,k}) + \eta v_{i,k}} - \eta  \|\hg_k\|^2 + \frac{1}{n}\iprod{\hg_k}{\eta \sum_{i=1}^n (d_{i,k} - v_{i,k})}, 
    \end{aligned}
    \ee
    where we use $\frac{1}{n}\sum_{i=1}^n d_{i,k} =  \hg_k$ and $\sum_{i=1}^n s_{i,k} = \sum_{i=1}^n \widetilde{x}_{i,k}$, which is derived from \eqref{eq:rela-s-x}. It follows from the Lipschitz-type inequality of $\Pcal_{\Mcal}$ in Lemma \ref{lemma-project} that
    \be \label{eq:b11} 
    \begin{aligned}
    & \iprod{\hg_k}{\sum_{i=1}^n (x_{i,k+1} - x_{i,k} - \gamma (s_{i,k} - \widetilde{x}_{i,k}) + \eta v_{i,k} } \\
    = & \iprod{\hg_k}{\sum_{i=1}^n \Pcal_{\Mcal}(x_{i,k} + \gamma (s_{i,k} - \widetilde{x}_{i,k}) - \eta v_{i,k}) - (x_{i,k} +\gamma (s_{i,k} - \widetilde{x}_{i,k}) -\eta v_{i,k}) } \\
    \leq & - \gamma \iprod{\hg_k}{\sum_{i=1}^n \Pcal_{N_{x_{i,k}}\Mcal}(s_{i,k} - \widehat{x}_{i,k})} + Q\|\hg_k\|\sum_{i=1}^n\| \gamma(s_{i,k} - \widetilde{x}_{i,k}) + \eta v_{i,k} \|^2 \\
    \leq & - \gamma \iprod{\hg_k}{\sum_{i=1}^n \Pcal_{N_{x_{i,k}}\Mcal}(s_{i,k} - \widehat{x}_{i,k})}  + QL (2\gamma^2 \|\bs_k - \tbx_k\|^2  + 2\eta^2 \| \bd_k \|^2)  \\
    = & \gamma \iprod{\hg_k}{\sum_{i=1}^n \Pcal_{N_{x_{i,k}}\Mcal}(\widehat{x}_{i,k} - s_{i,k} )} + 2QL(\gamma^2 \|(\bW -I)\tbx_k\|^2 + \eta^2 \|\bd_k - \hbd_k\|^2 + \eta^2 \|\hbd_k\|^2)  \\
    \leq & \gamma \iprod{\hg_k}{\sum_{i=1}^n \Pcal_{N_{x_{i,k}}\Mcal}(\widehat{x}_{i,k} - s_{i,k} )} + 16 QL\gamma^2( \|\tbx_k - \bx_k\|^2 + \|\bx_k - \bbx_k\|^2)\\
    & + 2QL\eta^2  (\|\bd_k - \hbd_k\|^2 + n \|\hg_k\|^2),
    \end{aligned}
    \ee
    where the second inequality is from Young's inequality and $\|\widehat{g}_k\| \leq \frac{1}{n}\sum_{i=1}^n\|\nabla f_i(x_{i,k})\|\leq L$, and the third inequality is due to $\sigma_{\max} (W - I)\leq 2$ and Young's inequality. Note that
    \be \label{eq:normal1}
    \begin{aligned}
        & \gamma \iprod{\hg_k}{\sum_{i=1}^n \Pcal_{N_{x_{i,k}}\Mcal}(\widehat{x}_{i,k} - s_{i,k})} =   \gamma \sum_{i=1}^n \iprod{\hg_k - \grad f_i(x_{i,k})}{ \Pcal_{N_{x_{i,k}}\Mcal}(\widehat{x}_{i,k} - s_{i,k})} \\
        \leq & \frac{1}{4L} \sum_{i=1}^n\| \hg_k - \grad f_i(x_{i,k}) \|^2 + \gamma^2 L \sum_{i=1}^n   \|\Pcal_{N_{x_{i,k}}\Mcal}(\widehat{x}_{i,k} - s_{i,k})\|^2 \\
        \leq & \frac{1}{4nL}\sum_{i=1}^n \sum_{j=1}^n\|\grad f_i(x_{i,k}) - \grad f_j(x_{j,k}) \|^2 + \gamma^2 L \| (\bW -I) \tbx_k\|^2 \\
        \leq & \frac{L}{4}\|\bx_k - \bbx_k\|^2 + 8\gamma^2 L(\|\bx_k - \tbx_k\|^2 + \|\bx_k - \bbx_k\|^2), 
    \end{aligned}
    \ee
    where the first inequality is from Young's inequality and the last inequality is due to $\sigma_{\max}(W- I) \leq 2$. Similarly, since $d_{i,k} - v_{i,k} = \Pcal_{N_{x_{i,k}} \Mcal}(d_{i,k})$, we have
    \be \label{eq:normal2} 
    \begin{aligned}
        \eta \iprod{\hg_k}{\sum_{i=1}^n (d_{i,k} - v_{i,k})} = & \leq \eta \sum_{i=1}^n \iprod{\hg_k - \grad f_i(x_{i,k})}{d_{i,k} - v_{i,k}} \\
        \leq & \frac{L}{4}\|\bx_k - \bbx_k\|^2 + \eta^2 L(\|\bd_k - \hbd_k\|^2 +  n\|\hg_k\|^2). 
    \end{aligned}
    \ee
    Plugging \eqref{eq:b11}, \eqref{eq:normal1}, and \eqref{eq:normal2} into \eqref{eq:b1} gives
    \be \label{eq:b1-f} 
    \begin{aligned}
        b_1 \leq & -(\eta - 2\eta^2 QL - \eta^2 L)  \|\hg_k\|^2 + \frac{16 QL\gamma^2 + 8 \gamma^2 L + L}{n} \| \bx_k - \bbx_k\|^2 \\
        &+ \frac{16 QL\gamma^2 + 8 \gamma^2 L}{n} \|\bx_k - \tbx_k\|^2 + \frac{2QL\eta^2 + \eta^2 L}{n}\|\bd_k - \hbd_k\|^2. 
    \end{aligned}
    \ee
    Regarding $b_2$, it follows from Lemma \ref{lemma:quadratic} that 
    \be \label{eq:b2} 
    \begin{aligned}
        b_2 \leq & \frac{M_2^2 L}{n\eta^2 }(\|\bx_k - \bbx_k\|^4 + \|\bx_{k+1} - \bbx_{k+1}\|^4). 
    \end{aligned}
    \ee
    Plugging $\Omega_k^4 \leq \frac{192n L_g^2}{\gamma^2 \rho^2}$ from Lemma \ref{lem:bound} into \eqref{recur-1} and \eqref{recur-3} gives
    \bee \begin{aligned}
        \Omega_1^{k+1} \leq &\left(1-\frac{\alpha}{2} + \frac{96 \gamma^2}{\alpha} \right) \Omega_1^k
           + \frac{96 \gamma^2}{\alpha} \Omega_3^k 
           + \frac{24 \eta^2}{\alpha} \frac{192n L_g^2}{\gamma^2 \rho^2} + \frac{24 \eta^2}{\alpha} n L_g^2,\\
        \Omega_3^{k+1}  \leq  &\frac{64}{\hat{\rho}}\gamma^2 \Omega_1^k + \left( 1- \frac{\hat{\rho}}{2}\right) \Omega_3^k + \frac{16}{\hat{\rho}} \eta^2 \frac{192n L_g^2}{\gamma^2 \rho^2} + \frac{16}{\hat{\rho}} \eta^2 n L_g^2.  \\
    \end{aligned}
    \eee
    Based on the above inequalities and using the proof by induction, it is not difficult to verify that for small enough but fixed $C_{\gamma}$ and $C_{\eta}$,   
    \[ \|\bx_k - \bbx_k\|^2 \leq \frac{c_0\eta^2}{\gamma^4},  \]
    where $c_0 > 0$ is a constant independent of $k$, $\eta$ and $\gamma$. Then, we have
    \be \label{eq:b2-f} 
    \begin{aligned}
        b_2 \leq & \frac{c_0M_2^2 L}{n\gamma^4 }(\|\bx_k - \bbx_k\|^2 + \|\bx_{k+1} - \bbx_{k+1}\|^2). 
    \end{aligned}
    \ee
    For $b_3$, we have 
    \be \label{eq:b3} 
    \begin{aligned}
        b_3 \leq & \frac{3L}{4} \times 4\|\hbx_{k+1} - \hbx_{k}\|^2 = \frac{3L}{n} \|\sum_{i=1}^n (x_{i,k+1} - x_{i,k})\|^2  \\
        = & \frac{3L}{n} \| \sum_{i=1}^n \left(\Pcal_{\Mcal}(x_{i,k} + \gamma (s_{i,k} - \widetilde{x}_{i,k}) - \eta v_{i,k}) - x_{i,k} \right) \|^2 \\
        \leq & \frac{12L}{n} \sum_{i=1}^n \| \gamma (s_{i,k} - \widetilde{x}_{i,k}) - \eta v_{i,k} \|^2 
        \leq  \frac{24 L}{n}(\| \gamma (\bW - I) \tbx_k \|^2 + \eta^2 \|\bd_k\|^2) \\
        \leq & \frac{24L}{n}\left(8 \gamma^2 (\|\bx_k - \tbx_k  \|^2 + \|\bx_k - \bbx_k\|^2) + \eta^2 (\|\hbd_k\|^2 + \|\bd_k - \hbd_k\|^2) \right) \\
        = & \frac{192L\gamma^2}{n} (\|\bx_k - \tbx_k  \|^2 + \|\bx_k - \bbx_k\|^2) + \frac{24L\eta^2}{n} ( n\|\hg_k\|^2 + \|\bd_k - \hbd_k\|^2),
    \end{aligned}\ee
    where we use the 2-Lipschitz continuity in the first and second inequalities, the third inequality is due to $\|v_{i,k}\| \leq \|d_{i,k}\|$, and the fourth inequality is from $\sigma_{\max}(W - I) \leq 2$. 

    Plugging \eqref{eq:b1-f}, \eqref{eq:b2-f}, and \eqref{eq:b3} into \eqref{eq:dec-lip} yields \eqref{eq:descent}.   
\end{proof}

\begin{theorem}
    Suppose that Assumptions \ref{assum-w} and \ref{ass:lip} hold. Let $\gamma = C_{\gamma} \alpha$ and $\eta = C_{\eta}\gamma \rho/L$.
    For sufficiently small $C_{\eta}$ and $C_{\gamma}$, we have
    \be \label{eq:complexity} 
    \frac{1}{K}\sum_{k=1}^K \|\hg_k\|^2 \leq \frac{f(\bx_0) - \inf_{\bx} f + M}{cK},
    \ee
    where $c \in (0,1)$ and $M>0$ are two constants. 
\end{theorem}
\begin{proof}
Note that only the term $- (\eta - (2 M +  25)L\eta^2)  \|\hg_k\|^2$ is negative in the right handside of \eqref{eq:descent}. To establish the convergence, we suffice to show that there exist positive constants $c_1, e_1,  c_2, e_2,  c_3, e_3, c_4, e_4$ independent of $\eta$ such that
\be \label{eq:iso-grad} \begin{aligned}
    & \sum_{k=1}^K \|\bx_k - \bbx_k\|^2 \leq c_1 \eta^2 \sum_{k=1}^{K-1} \|\hg_k\|^2 + e_1, \; \sum_{k=1}^K \|\bx_k - \tbx_k\|^2 \leq c_2 \eta^2 \sum_{k=1}^{K-1} \|\hg_k\|^2 + e_2,
    \\
    & \sum_{k=1}^K \|\bd_k - \hbd_k\|^2 \leq c_3 \sum_{k=1}^{K-1} \eta^2 \|\hg_k\|^2 + e_3, 
    \; \sum_{k=1}^K \|\bd_k - \tbd_k\|^2 \leq c_4 \sum_{k=1}^{K-1} \eta^2 \|\hg_k\|^2 + e_4.
\end{aligned}
\ee

The proof can be proceeded by induction. It is easy to see that \eqref{eq:iso-grad} holds for large enough $e_i, i=1,\dots 4$. Assume that \eqref{eq:iso-grad} holds at $K$.  
Let us give the fact that: for any positive sequences $\{a_k\}, \{b_k\}$ satisfying 
\[ a_{k+1} \leq \zeta a_k + b_k \]
with $\zeta \in (0,1)$, it holds
\[ \sum_{k=1}^K a_k \leq \frac{1}{1-\zeta} \sum_{k=1}^{K-1} b_{k} + \frac{a_1}{1-\zeta}.  \]
Applying the above inequality to \eqref{recur-1}, \eqref{recur-2}, \eqref{recur-3}, and \eqref{recur-4} yields
\[ \begin{aligned}
    \sum_{k=1}^{K+1} \Omega_1^k \leq & \frac{4}{\alpha} \sum_{k=1}^{K} \left( \frac{96 \gamma^2}{\alpha} \Omega_3^k 
           + \frac{24 \eta^2}{\alpha} \Omega_4^k + \frac{24 n \eta^2}{\alpha}\|\hg_k\|^2  \right) + \frac{4 o_1}{\alpha} \\
    \sum_{k=1}^{K+1} \Omega_2^{k+1} \leq  &\frac{4}{\alpha} \sum_{k=1}^K \left( \frac{288 L^2\gamma^2}{\alpha} \Omega_1^k + \frac{288 L^2 \gamma^2}{\alpha} \Omega_3^k  + \frac{24 \gamma^2 + 72L^2 \eta^2} {\alpha} \Omega_4^k  + \frac{72 L^2 \eta^2}{\alpha} \Omega_5^k \right) + \frac{4 o_2}{\alpha}  \\
      \sum_{k=1}^{K+1}  \Omega_3^{k+1}  \leq  & \frac{2}{\hat{\rho}} \sum_{k=1}^K \left( \frac{64}{\hat{\rho}}\gamma^2 \Omega_1^k  + \frac{16}{\hat{\rho}} \eta^2 \Omega_4^k + \frac{16n \eta^2}{\hat{\rho}} \|\hg_k\|^2 \right) + \frac{2o_3}{\hat{\rho}} \\
        \sum_{k=1}^{K+1}  \Omega_4^{k+1} \leq & \frac{4}{\gamma \rho} \sum_{k=1}^K \left( \frac{288 L^2\gamma}{\rho} \Omega_1^k +   \frac{24\gamma}{\rho} \Omega_2^k 
        +  \frac{288 L^2\gamma}{\rho} \Omega_3^k +
        \frac{72 n L^2\eta^2}{\gamma \rho } \|\hg_k\|^2 \right) + \frac{4o_4}{\gamma\rho},       
\end{aligned}
\]
where $o_i, i=1,\dots,4$ are finite constants independent of $k$. Then, there exists sufficiently small but fixed $C_{\gamma}$ and $C_{\eta}$ such that \eqref{eq:iso-grad} holds at $K+1$. 

Summing \eqref{eq:descent} over $k=0, \dots,K$ and plugging \eqref{eq:iso-grad} lead to 
\[ c \eta \sum_{k=1}^K \|\hg_k\|^2 \leq f(\bx_0) - f(\bx_{k+1}) + M, \]
where $c \in (0,1), M> 0$ are finite constants. This implies \eqref{eq:complexity}.

\end{proof}

\section{Numerical experiments}
In this section, we present the numerical comparisons among our proposed DPRGC, the decentralized projected Riemannian gradient tracking method (DPRGT) \cite{deng2023decentralized}, and the decentralized Riemannian gradient tracking method (DRGTA) \cite{chen2021decentralized}. We set $\gamma =1$ for DPRGC in all experiments. 

Consider the decentralized principal component analysis problem:
\be \label{prob:pca}
\min _{\mathbf{x} \in \mathcal{M}^{n}}-\frac{1}{2 n} \sum_{i=1}^{n} \text{tr}(x_{i}^{\top} A_{i}^{\top} A_{i} x_{i}), \quad \text { s.t. } \quad x_{1}=\ldots=x_{n},
\ee
where $\Mcal^n:=\underbrace{{\rm St}(d,r) \times \cdots \times {\rm St}(d,r)}_{n}$, $A_{i} \in \mathbb{R}^{m_{i} \times d}$ is the local data matrix in $i$-th agent, and $m_{i}$ is the sample size.
Note that for any solution $x^*$ of \eqref{prob:pca}, $x^*Q$ with an orthogonal matrix $Q \in \R^{r\times r}$ is also a solution. We use the function
\[ d_s(x, x^*) := \min_{Q\in \R^{r\times r},\; Q^\top Q = QQ^\top = I_d} \; \|xQ - x^*\| \]
to compute the distance between two points $x$ and $x^*$.

\subsection{Synthetic dataset}
We fix $m_{1}=\ldots=m_{n}=1000, d=10$, and $r=5$. We then generate a matrix $B \in \R^{1000n  \times d}$ and do the singular value decomposition
\[ B = U \Sigma V^\top, \]
where $U \in \R^{1000n \times d}$ and $V \in \R^{d\times d}$ are orthogonal matrices, and $\Sigma \in \R^{d \times d}$ is a diagonal matrix. To control the distributions of the singular values, we set $\tilde{\Sigma} = {\rm diag} (\xi^j)$ with $\xi \in (0,1)$. Then, $A$ is set as
\[ A = U \tilde{\Sigma} V^\top \in \R^{1000n \times d}. \]
$A_i$ is obtained by randomly splitting the rows of $A$ to $n$ subsets with equal cardinalities. It is easy to check the first $r$ columns of $V$ form the solution of \eqref{prob:pca}. In the experiments, we set $\xi$ and $n$ to $0.8$ and $8$, respectively.

\begin{figure}[htp]
	\centering
	\includegraphics[width = 0.45 \textwidth]{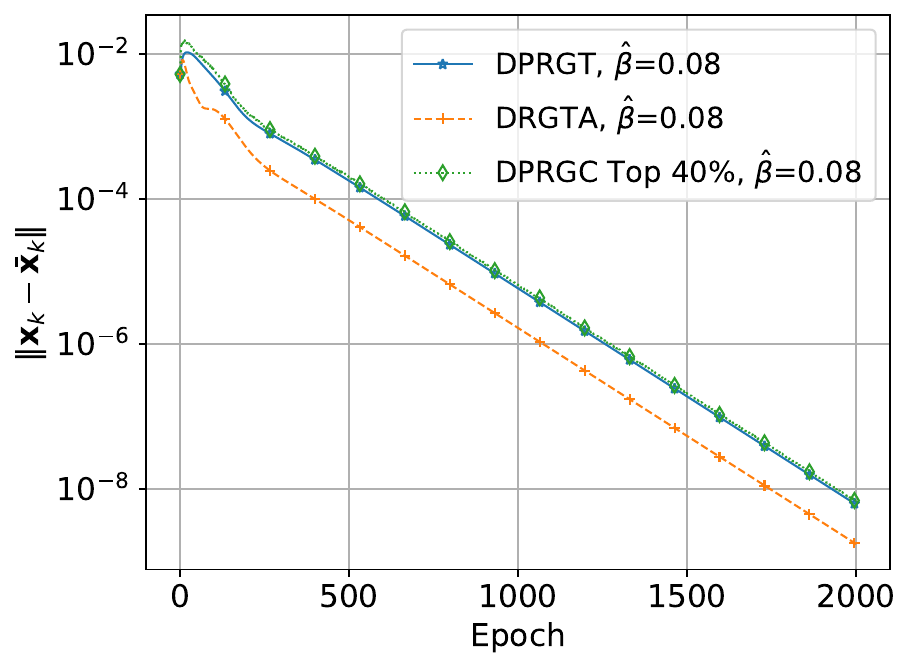}
	\includegraphics[width = 0.45 \textwidth]{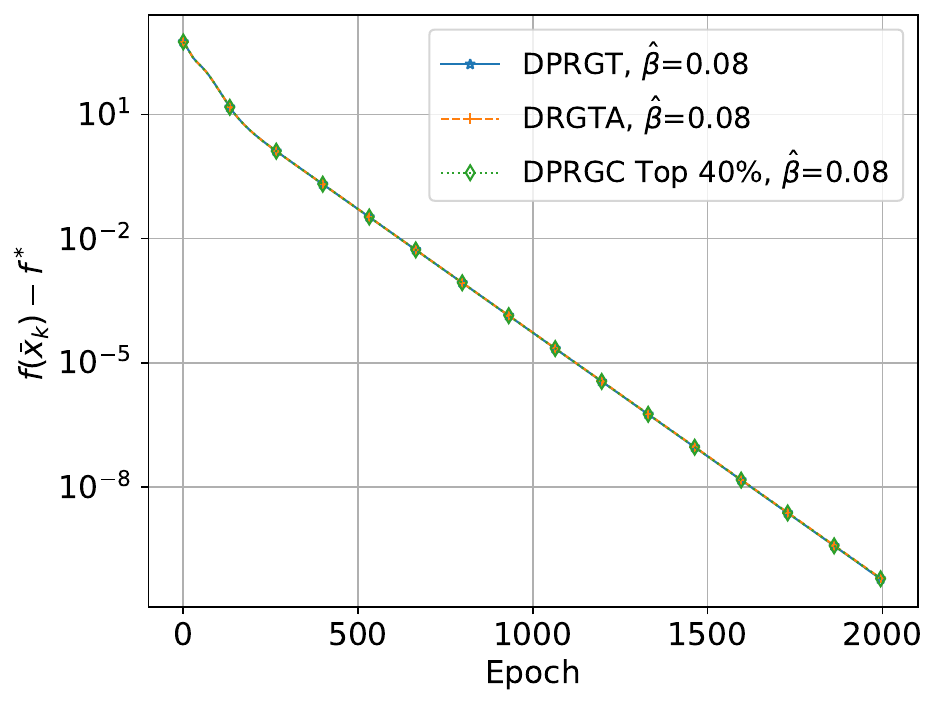} \\
	\includegraphics[width = 0.45 \textwidth]{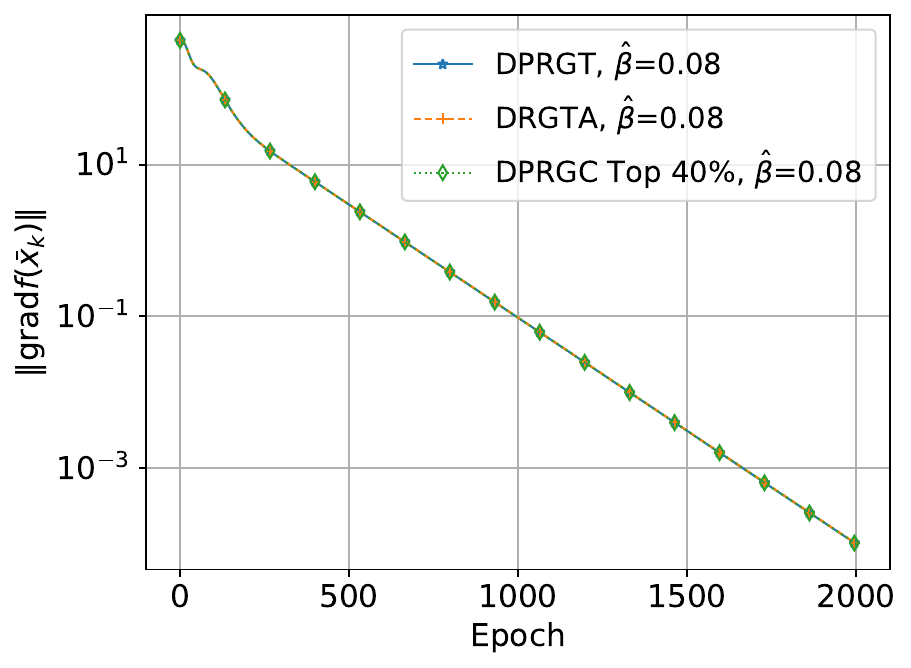}
	\includegraphics[width = 0.45 \textwidth]{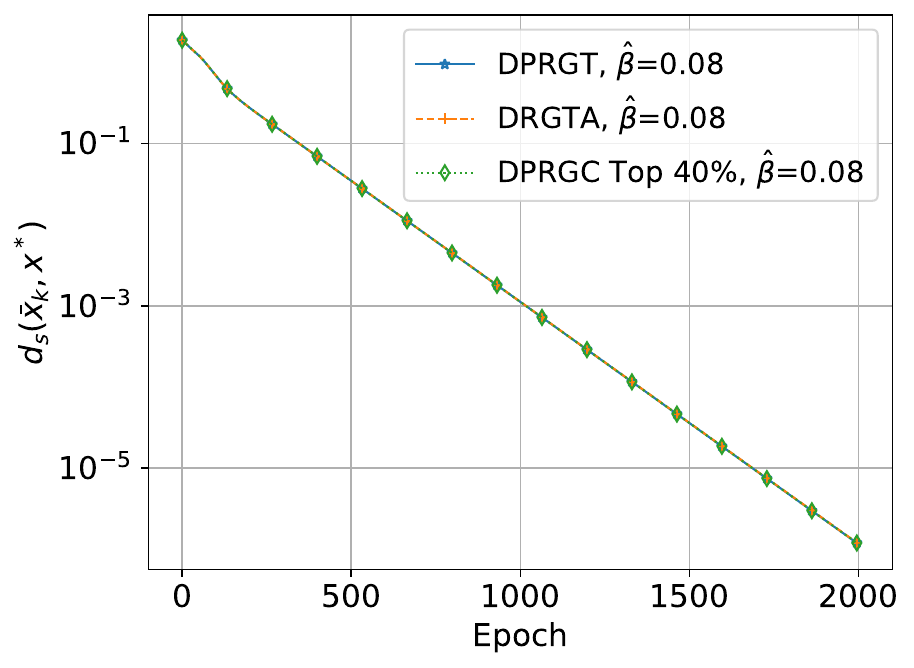}
	\caption{Iterative curves of different algorithms with respect to epochs on a synthetic dataset. DPRGT coincides with DPRGC when there is no compression.}	
	\label{fig:num-pca-alg-epoch}
\end{figure}

\begin{figure}[htp]
	\centering
	\includegraphics[width = 0.45 \textwidth]{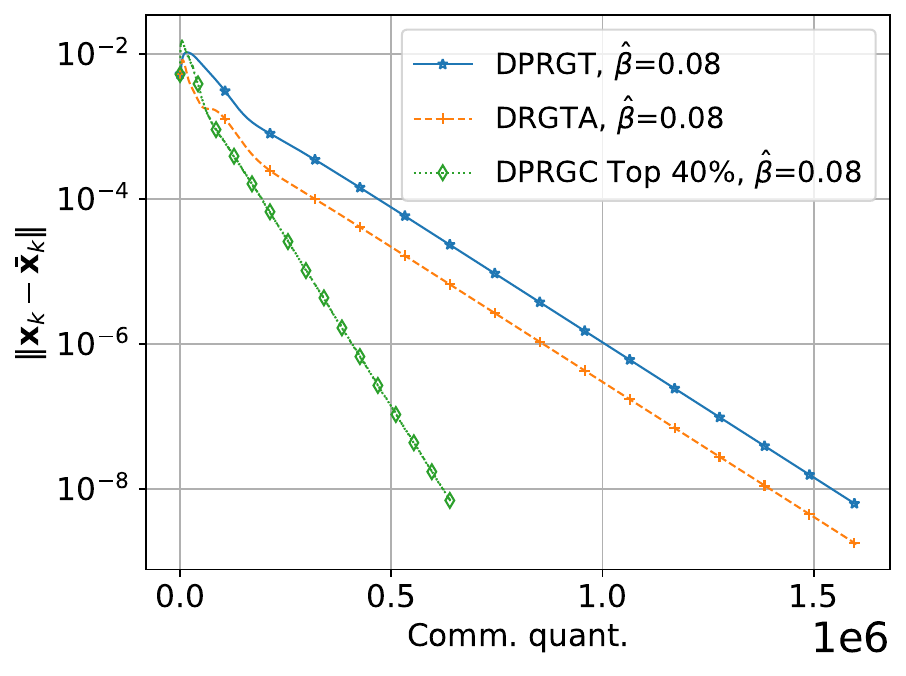}
	\includegraphics[width = 0.45 \textwidth]{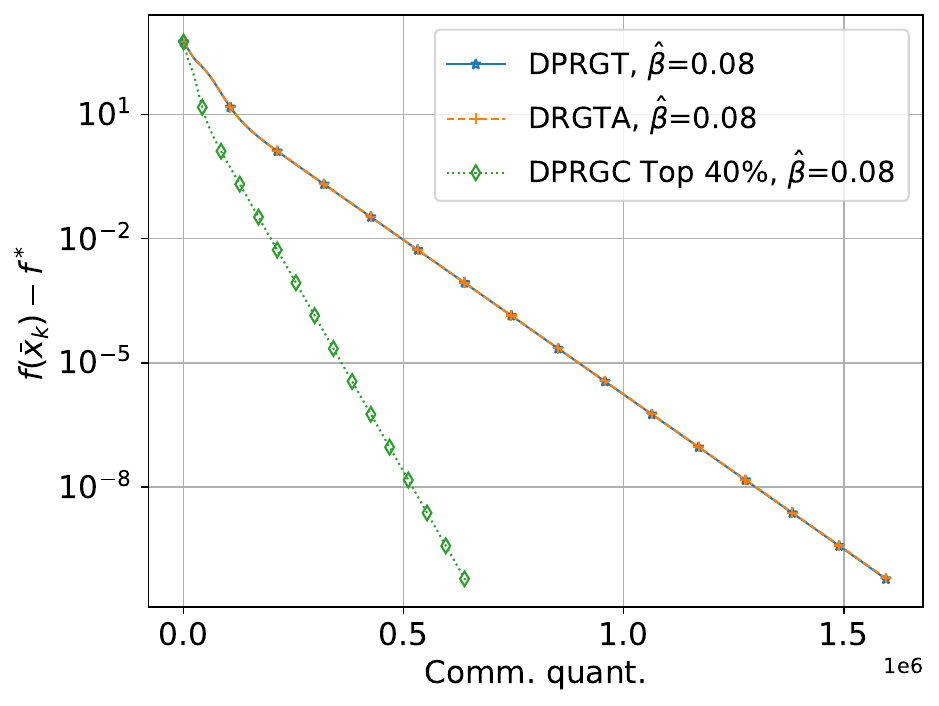} \\
	\includegraphics[width = 0.45 \textwidth]{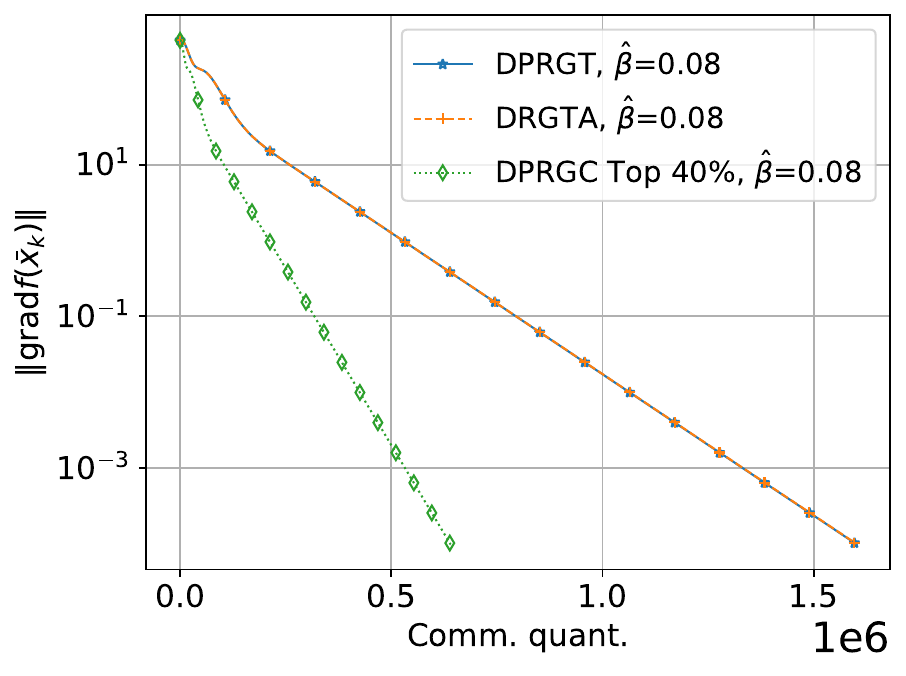}
	\includegraphics[width = 0.45 \textwidth]{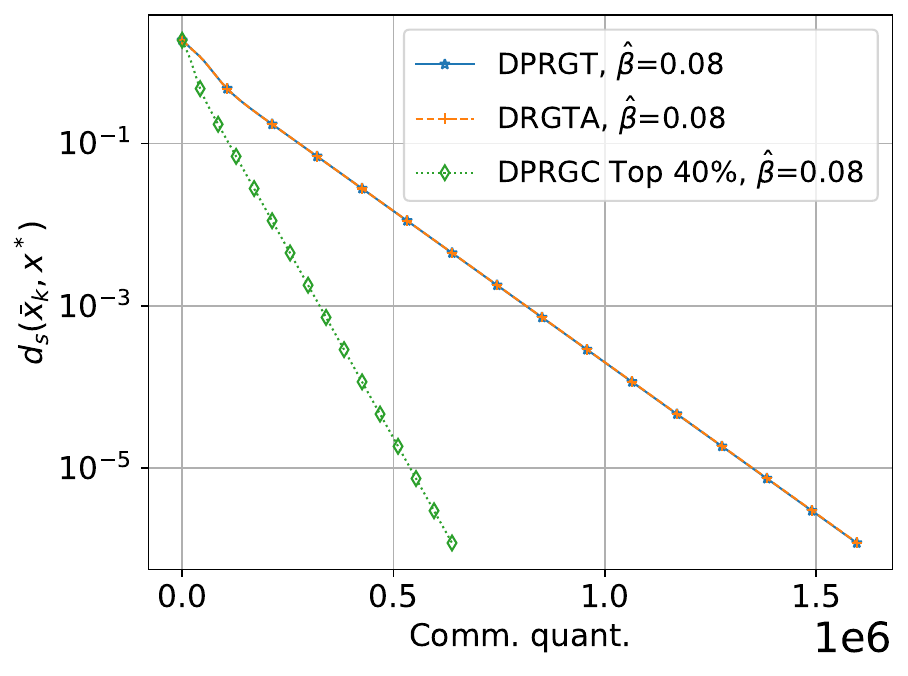}
	\caption{Iterative curves of different algorithms with respect to communication quantities on a synthetic dataset. DPRGT coincides with DPRGC when there is no compression.}	
	\label{fig:num-pca-alg-quant}
\end{figure}

We employ fixed step sizes for all algorithms. For all algorithms, we use the step size $\alpha=\frac{\hat{\beta}n}{\sum_{i=1}^n m_i}$. The grid search is utilized to find the best $\hat{\beta}$. We choose the polar decomposition as the retraction operator for DRGTA. We test several graph matrices to model the topology across the agents, namely, the Erdos-Renyi (ER) network with probability $p = 0.3, 0.6$, and the Ring network. Throughout this section, we select the mixing matrix $W$ to be the Metropolis constant edge weight matrix \cite{shi2015extra}. As single-step consensus is able to guarantee the convergence, we use $W$ instead of multiple-step consensus $W^t$ ($t > 1$) for all algorithms.

The results of different algorithms are presented in Figures \ref{fig:num-pca-alg-epoch} and \ref{fig:num-pca-alg-quant}. Note that DPRGT can be seen as a special case of DPRGC, where $\gamma = 1$ and no compression is used (i.e., $\mathcal{C}$ is the identity mapping). The iterative curves of different algorithms with respect are quite close. However, DPRGC has significant improvement in terms of the communication quantities, which is measured by the total number of entries communicated in all iterations. Here, the compression operator $\mathcal{C}(x)$ keeps the $40\%$ largest entries of each column of $x$. Figure \ref{fig:num-pca-alg-quant} shows that the same accuracies can be achieved by DPRGC with only $40\%$ communication costs.

\begin{figure}[htp]
	\centering
    \includegraphics[width = 0.45 \textwidth]{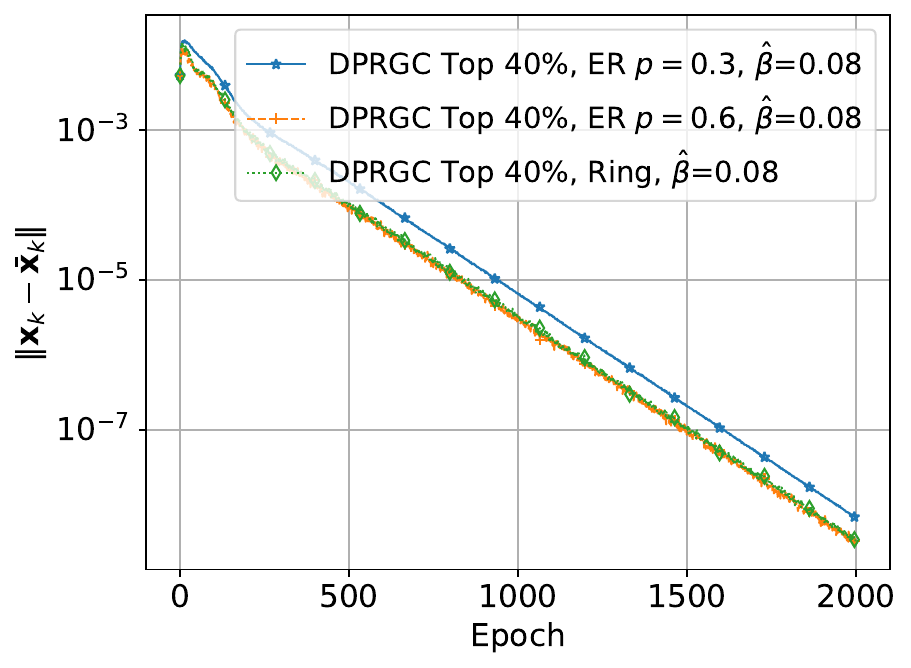}
    \includegraphics[width = 0.45 \textwidth]{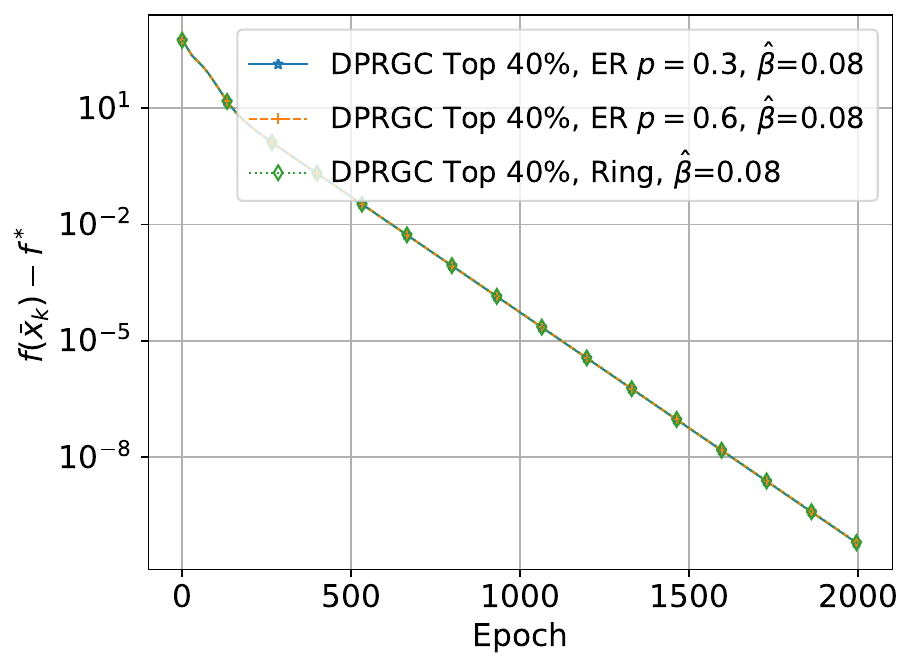} \\
    \includegraphics[width = 0.45 \textwidth]{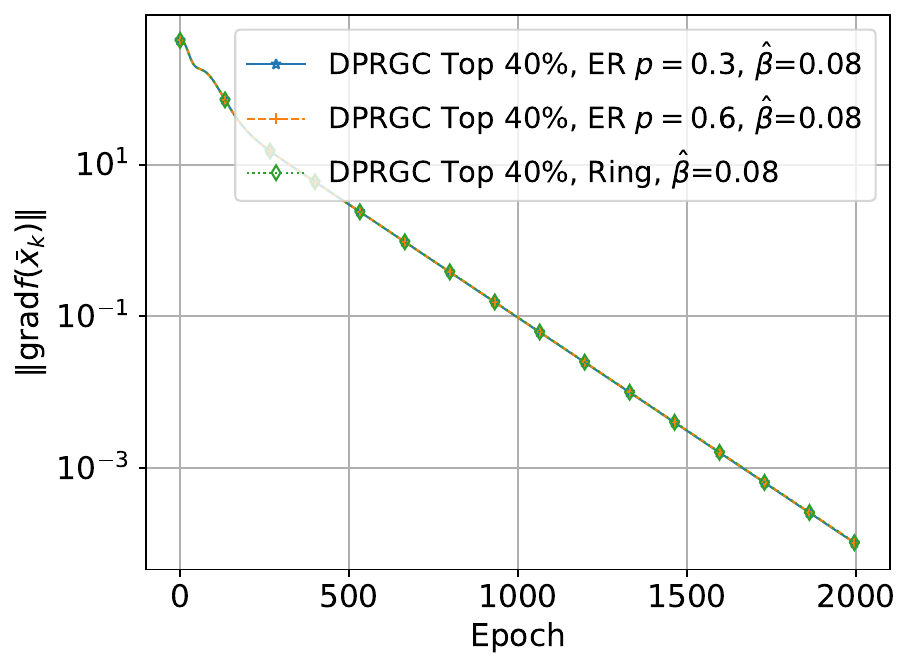}
    \includegraphics[width = 0.45 \textwidth]{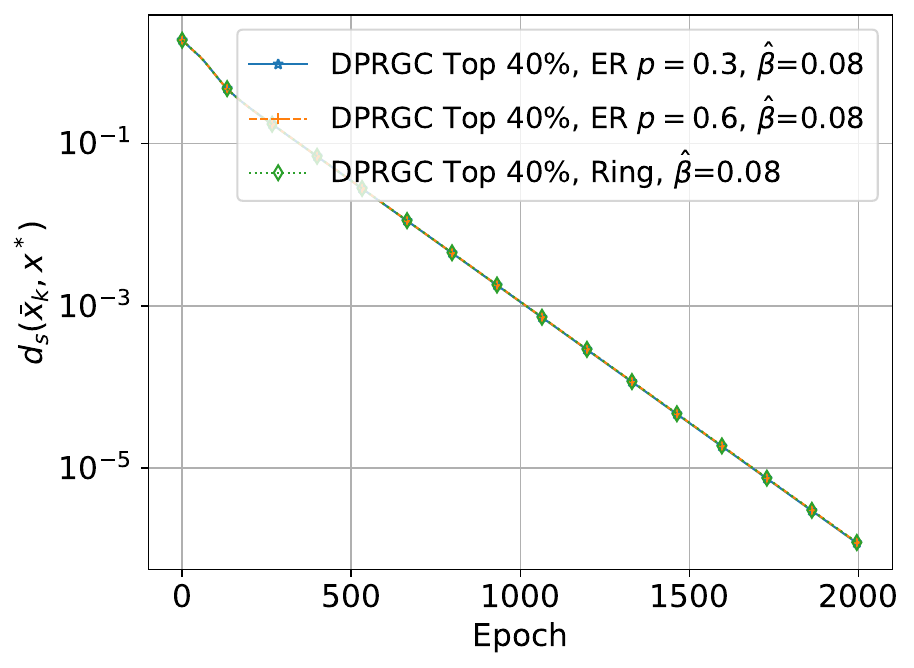}
	\caption{Iterative curves of DPRGC with different graph networks on a synthetic dataset. DPRGT coincides with DPRGC when there is no compression.}	
	\label{fig:num-pca-alg}
\end{figure}

\begin{figure}[htp]
	\centering
	\includegraphics[width = 0.45 \textwidth]{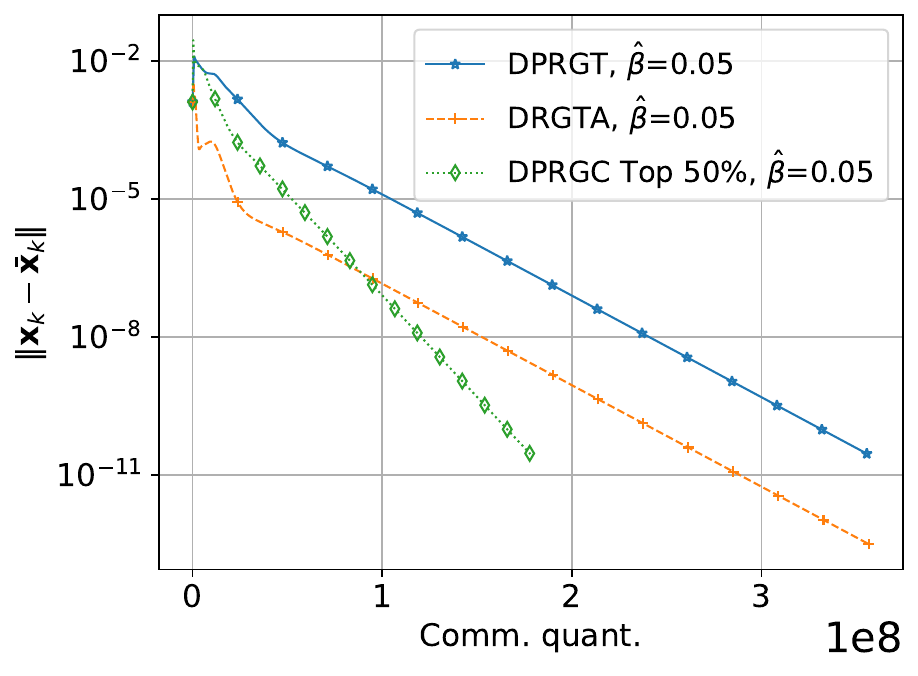}
	\includegraphics[width = 0.45 \textwidth]{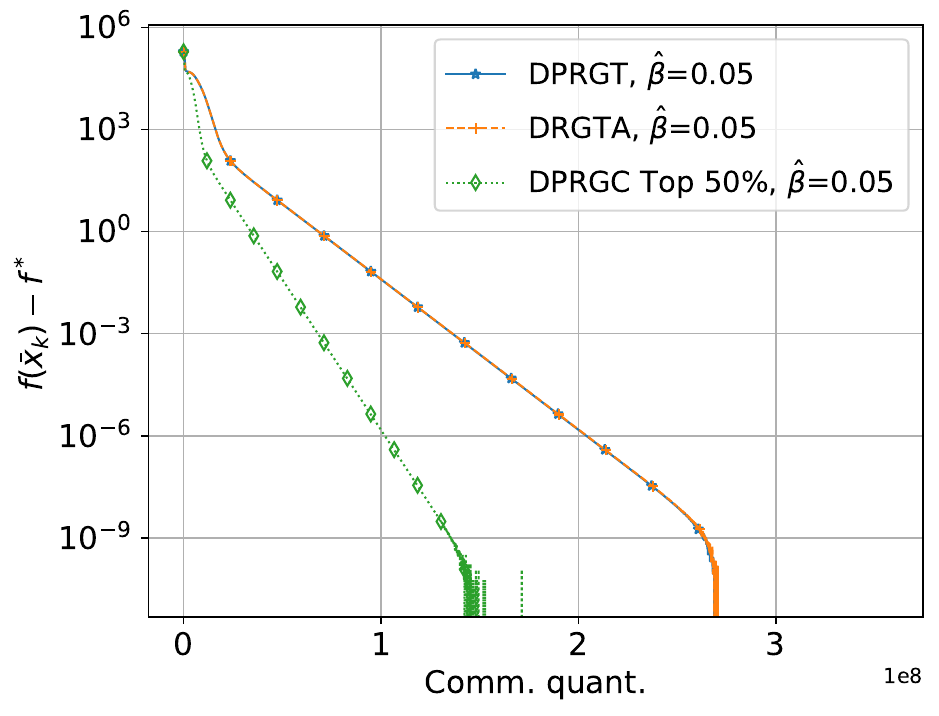} \\
	\includegraphics[width = 0.45 \textwidth]{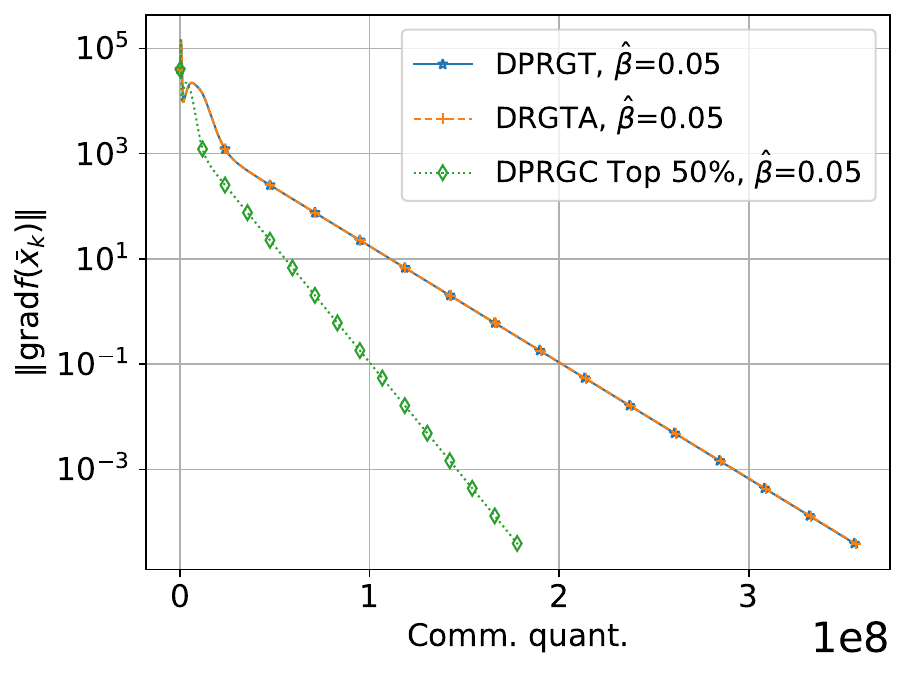}
	\includegraphics[width = 0.45 \textwidth]{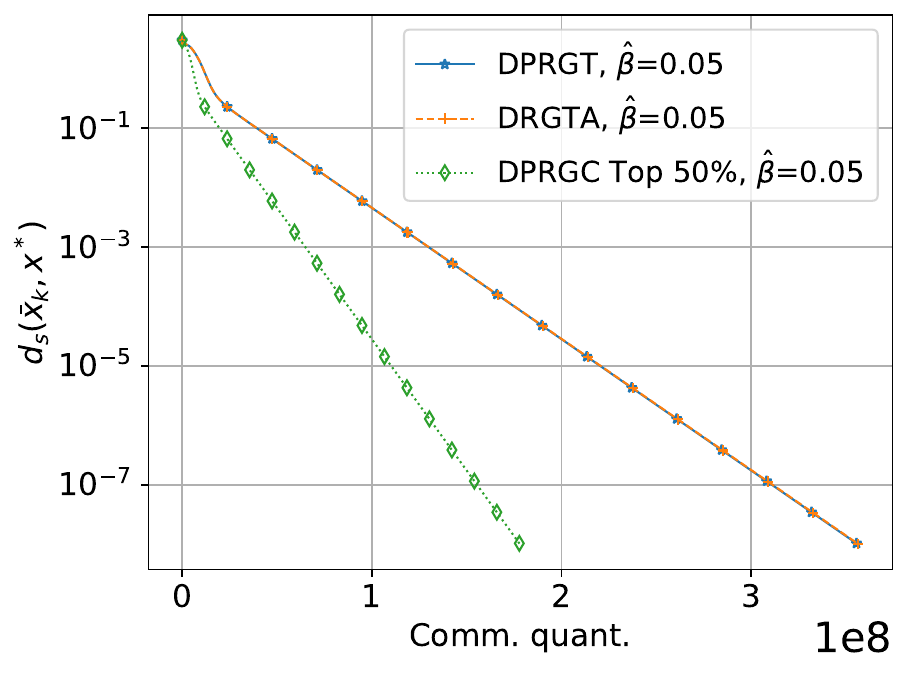}
	\caption{Iterative curves of different algorithms with respect to epochs on the Mnist dataset. DPRGT coincides with DPRGC when there is no compression.}
	\label{fig:num-pca-alg-quant-mnist}
\end{figure}

We also test the impacts of different graph networks in Figure \ref{fig:num-pca-alg}. It can be seen that there is not much difference in the iterative curves of DPRGC when using different networks. A dense graph leads to a fast convergence speed of the consensus.

\subsection{Mnist dataset}
To evaluate the efficiency of our proposed method, we also perform numerical tests on the Mnist dataset \cite{lecun1998mnist}. The testing images consist of 60000 handwritten images of size $32 \times 32$ and are used to generate $A_i$'s. We first normalize the data matrix by dividing 255 and randomly split the data into $n=8$ agents with equal cardinality. Then, each agent holds a local matrix $A_i$ of dimension $\frac{60000}{n} \times 784$. We compute the first 5 principal components, i.e., $d=784, r=5$.

For all algorithms, we use the fixed step sizes $\alpha = \frac{\hat{\beta}}{60000}$ with a best-chosen $\hat{\beta}$.
Similar to the above setting, we see from Figure \ref{fig:num-pca-alg-quant-mnist} that DPRGC converges to a point with similar accuracy by only using half communication costs.

\section{Conclusion}
Our study addresses the significant challenge of decentralized optimization on compact manifolds, focusing on the issue of multi-step consensus and communication efficiency in projection/retraction-type algorithms. By investigating the smoothness structure and the asymptotic 1-Lipschitz continuity of manifold constraints, we successfully demonstrate that single-step consensus is both feasible and effective. Our findings are supported by the development of a novel communication-efficient gradient algorithm, which incorporates communication compression to minimize per-iteration communication demands. We also establish that our method achieves an iteration complexity of $\mathcal{O}(\epsilon^{-1})$, compatible with the Euclidean framework. Through extensive numerical experiments, our approach is shown to outperform existing state-of-the-art methods, highlighting its efficacy and potential for practical deployment in decentralized manifold optimization scenarios.

\bibliographystyle{siamplain}
\bibliography{ref}



\end{document}